\documentclass[12pt]{amsart}
\usepackage[english]{babel}
\usepackage[hmargin=2.5cm,bmargin=2.5cm,tmargin=3cm]{geometry}

\usepackage{amsmath,amsthm,amssymb,amsfonts,latexsym}

\usepackage{academicons, xcolor}
\definecolor{orcidlogocol}{HTML}{A6CE39}

\usepackage{hyperref,enumerate}

\DeclareMathOperator{\dist}{dist}
\DeclareMathOperator{\diam}{diam}
\DeclareMathOperator{\ran}{Ran}

	 \newcommand{\obs}{\mathrm{obs}}
	 \newcommand{\euler}{\mathrm{e}}
     \newcommand{\CC}{\mathbb{C}}
     \newcommand{\EE}{\mathbb{E}}
     \newcommand{\NN}{\mathbb{N}}
     \newcommand{\PP}{\mathbb{P}}
     \newcommand{\RR}{\mathbb{R}}
     \newcommand{\ZZ}{\mathbb{Z}}
     
     \newcommand{\vol}{\operatorname{Vol}}     
     \newcommand{\Csob}{C_{\mathrm{sob}}}
     \newcommand{\drm}{\mathrm{d}}
	 \newcommand{\tildel}{\tilde \partial}

\newtheorem{theorem}{Theorem}
\newtheorem{lemma}[theorem]{Lemma}
\newtheorem{definition}[theorem]{Definition}
\newtheorem{proposition}[theorem]{Proposition}

\newtheorem{remark}[theorem]{Remark}
\newtheorem{corollary}[theorem]{Corollary}

\makeatletter
\@namedef{subjclassname@2020}{%
  \textup{2020} Mathematics Subject Classification}
\makeatother

\title[Magnetic Bernstein inequalities and spectral inequality]{Magnetic Bernstein inequalities and spectral inequality on thick sets for the Landau operator} 



\author[]{Paul Pfeiffer}
\address{Paul Pfeiffer, Ludwigs-Maximilians-Universit\"at München, Fakultät für Mathematik, Informatik und Statistik, 80333 Munich, Germany,
\href{https://orcid.org/0000-0001-8538-5314}{orcid.org/0000-0001-8538-5314}}
\email{pfeiffer@math.lmu.de}

\author[]{Matthias T\"aufer}
\address{Matthias T\"aufer, Universit\'e Polytechnique Hauts-de-France, C\'ERAMATHS/DMATHS,
59313 Valenciennes, France,
\href{https://orcid.org/0000-0001-8473-2310}{orcid.org/0000-0001-8473-2310}}
\email{matthias.taufer@uphf.fr}

\thanks{Much of this work was done while both authors were employed by FernUniversität in Hagen, Germany. The authors thank the anonymous referees for numerous helpful comments.
}

\begin{document}
\begin{abstract}
	We prove a \emph{spectral inequality} for the Landau operator.
	This means that for all $f$ in the spectral subspace corresponding to energies up to $E$, the $L^2$-integral over suitable $S \subset \RR^2$ can be lower bounded by an explicit constant times the $L^2$-norm of $f$ itself.
	We identify the class of all measurable sets $S \subset \RR^2$ for which such an inequality can hold, namely so-called \emph{thick} or \emph{relatively dense} sets, and deduce an asymptotically optimal expression for the constant in terms of the energy, the magnetic field strength and in terms of parameters determining the thick set $S$.
	Our proofs rely on so-called magnetic Bernstein inequalities.
	
	As a consequence, we obtain the first proof of null-controllability for the magnetic heat equation (with sharp bound on the control cost), and can relax assumptions in existing proofs of Anderson localization in the continuum alloy-type model.
\end{abstract}

\maketitle

\section{Introduction}

The \emph{Landau operator}
\[
	H_B
	:=
	\left( 
	i \nabla 
	+
	\frac{B}{2}
	\begin{pmatrix}
		- x_2
		\\
		x_1
	\end{pmatrix}
	\right)^2
\] 
 occasionally also called \emph{twisted Laplacian}, describes the motion of a particle in two dimensions, subject to a constant magnetic field of strength $B > 0$. Here, $x_1$ and $ x_2$ denote the operators of multiplication with the first and second coordinate, respectively.
 It is a self-adjoint operator the spectrum of which consists of infinitely degenerate eigenvalues at particular energies, namely $B, 3 B, 5 B, \dots$. 
 These energies are also referred to as the \emph{Landau levels}.
The Landau operator is relevant for a host of phenomena in Physics, including explanations for Landau diamagnetism~\cite{Landau-1930}, Hofstadter's butterfly~\cite{Hofstadter-76}, and von Klitzing's description of the quantized Hall effect~\cite{vonKlitzing-86}.

In this article, we prove optimal \emph{spectral inequalities}, that are lower bounds on the mass of functions, sampled on a subdomain $S \subset \RR^2$, uniform for all function in the spectral subspace below a given energy $E$
\begin{equation}
	\label{eq:spectral_inequality_intro}
	\lVert f \rVert_{L^2(\RR^2)}^2
	\leq
	C(E,B,S)
	\lVert f \rVert_{L^2(S)}^2
	\quad
	\text{for all}
	\quad
	f \in \ran \mathbf{1}_{(- \infty, E]}(H_B)
\end{equation}
where $\mathbf{1}_{(- \infty, E]}(H_B)$ is the spectral projector onto energies up to $E$ with respect to the self-adjoint operator $H_B$ and $\ran$ denotes the range of an operator.
Clearly, not for every $S \subset \RR^2$ such an inequality can hold.
We identify the necessary and sufficient criterion on $S \subset \RR^2$ for~\eqref{eq:spectral_inequality_intro} to hold, namely \emph{thickness} or \emph{relative density}.
Furthermore, we provide an explicit expression for the constant $C(E,B,S)$, and show that, in some sense, it is optimal in $E,B$, and parameters determining the thick set $S$. For details, see the remark below Theorem~\ref{thm:main}. 
In particular, for fixed $B$ and $S$, the constant $C(E,B,S)$ grows as $\exp (C \sqrt{E})$ for $E \to \infty$, which is essential for the applications to control theory. 

So far, examples of differential operators on $\RR^d$, for some $d \in \NN$, where the class of all measurable sets $S \subset \RR^d$ leading to a spectral inequality has been identified, are rare: 
One example is the free (negative) Laplacian $- \Delta$, where spectral inequalities can be inferred from the Kovrijkine-Logvinenko-Sereda theorem~\cite{Kovrijkine-00, EgidiV-18, WangWZZ-19}, another one is the harmonic Laplacian $- \Delta + \lvert x \rvert^2$ where explicit calculations are possible~\cite{BeauchardJPS-21, EgidiS-21}.
In a number of recent works, which spectral inequalities on thick sets are proven for perturbations of these operators~\cite{leBalchHM-24,Zhu-24,SuSY-25} whence the optimal criterion is also known in these cases.
Our main result, Theorem~\ref{thm:main}, now adds the Landau operator to this exclusive club.

Estimates as in~\eqref{eq:spectral_inequality_intro}, albeit without an explicit depencence of the constant $E$, have also been used in the context of Anderson localization for random Schr\"odinger operators where they are known as \emph{unique continuation principles}.
Even without the quantitative dependence of the constant on $E$ (which might give rise to further developments), our results yield immediate improvements of existing works since we no longer need to assume that $S$ is open, an ubiquitous technical assumption so far.
Another application of spectral inequalities are semiclassical estimates of restrictions of the Laplacian or the Landau operator onto bounded domains with Dirichlet or Neumann boundary conditions, see the recent~\cite{FrankLP-25} where our Theorem~\ref{thm:main} has been used. 

From a technical point of view, our main contribution are what we call \emph{magnetic Bernstein inequalities} (Theorem~\ref{thm:magnetic_Bernstein}).
Indeed, we are aware of two established strategies for proving spectral inequalities: On the one hand, the Kovrijkine-Logvinenko-Sereda theorem, on the other hand Carleman inequalities.
While the latter strategy offers more flexibility in terms of the choice of the operator, it usually requires the sampling set $S$ to be open.
The Kovrijkine-Logvinenko-Sereda theorem on the other hand crucially relies on so-called Bernstein inequalities which bound (the $L^2$-norm of) derivatives of functions in spectral subspaces to infer analyticity.
While almost trivial for the pure Laplacian, it turns out that in the case of the Landau operator, (ordinary) Bernstein inequalities no longer hold, see Remark~\ref{rem:bound_impossible}.
However, our workaround will be to work with covariant \emph{magnetic derivatives} and then use corresponding \emph{magnetic Bernstein inequalities} in $L^2$-norm to infer (ordinary) \emph{Bernstein-type inequalties in $L^1$-norm}.

The paper is organized as follows:
Section~\ref{sec:Definitions_and_results} contains definitions and our main results, namely the optimal spectral inequality for the Landau operator on $\RR^2$ (Theorem~\ref{thm:main}), as well as its restriction to boxes of finite volume (Theorem~\ref{thm:main_bounded_domain}).
In Section~\ref{sec:Magnetic_Bernstein_inequalities}, we prove the magnetic Bernstein inequalities (Theorem~\ref{thm:magnetic_Bernstein}), and Bernstein-type inequalities for the Landau operator in $L^1$-norm (Theorem~\ref{thm:magnetic_Bernstein_2}).
Also, Section~\ref{sec:Magnetic_Bernstein_inequalities} contains remarks and lemmas on optimality of our main results.
Section~\ref{sec:proofs} uses the Bernstein-type inequalities to prove Theorem~\ref{thm:main}.
In Section~\ref{sec:finite_volume}, we explain the necessary modifications for the finite-volume analogon.
Finally, Section~\ref{sec:applications} contains applications: Subsection~\ref{sec:heat} is about controllability and sharp control cost estimates for the magnetic heat equation (Theorems~\ref{thm:controlled_heat_1} and~\ref{thm:controlled_heat_2}) whereas Subsection~\ref{sec:RSO} contains applications to random Schr\"odinger operators, namely Wegner estimates, regularity of the integrated density of states, and Anderson localization in the continuum Anderson model in the case where the single-site potential is no longer assumed to be positive on an open, but merely on a measurable set.

\section{Definitions and main Results}
	\label{sec:Definitions_and_results} 

For $x = (x_1, x_2) \in \RR^2$, we denote by $\lvert x \rvert = (x_1^2 + x_2^2)^{1/2}$ its Euclidean norm and by $\lvert x \rvert_1 = \lvert x_1 \rvert + \lvert x_2 \rvert$ its $1$-norm.
The expression $\vol (S)$ refers to the Lebesgue measure of a measurable set $S \subset \RR^2$.
We will occasionally also use the one-dimensional Hausdorff measure of subsets of line segments in $\RR^2$ and denote the one-dimensional measure of such a set $T$ by $\vol_1 (T)$ for clarity.
For a measurable set $S$, $\mathbf{1}_S$ denotes its indicator function.
In particular, given a self-adjoint operator $A$, and $E \in \RR$, we denote by $\mathbf{1}_{(- \infty, E]}(A)$ the orthogonal projector onto the spectral subspace up to energy $E$, coresponding to $A$.
We write $C_0^\infty(\RR^2)$ for the space of smooth functions with compact support and $\mathcal{S}(\RR^2)$ for the space of Schwarz functions, that are smooth functions all derivatives of which decay faster at infinity than any polynomial.
We also denote by $\partial_i := \frac{\mathrm{d}}{\mathrm{d x_i}}$ the partial derivative with respect to the $x_i$ coordinate.

\begin{definition}
	\label{def:thick}
	Let $\ell = (\ell_1, \ell_2) \in (0,\infty)^2$, and $\rho  \in (0,1]$.
	A measurable set $S \subseteq \RR^2$ is called \emph{$(\ell, \rho)$-thick} if for every rectangle $Q$ with side lengths $(\ell_1, \ell_2)$, parallel to the axes, we have
	\[
	\vol ( S \cap Q )
	\geq
	\rho
	\vol ( Q )
	.
	\]
\end{definition}

If $S$ is $(\ell, \rho)$-thick for some $\ell, \rho$, it is also simply called \emph{thick}.
In the literature, one also finds the equivalent notion of \emph{relative dense} sets.
Thick sets seem to have originated in Fourier analysis~\cite{Paneah-61, Kacnelson-73, LogvinenkoS-74, Kovrijkine-00, Kovrijkine-01} but have attracted interest in the recent years~\cite{EgidiV-18, WangWZZ-19, LebeauM-19, EgidiV-20, MartinPS-20, BeauchardJPS-21, GreenJM-22, Taeufer-23, WangZ-23}.
\begin{definition}
	For $B > 0$ let 
	\[
	\tildel_1 = i \partial_1 -\frac{B}{2} x_2
	\quad
	\text{and}
	\quad 
	\tildel_2 = i \partial_2 + \frac{B}{2} x_1
	\]
	 be the \emph{magnetic derivatives at magnetic field strength $B$}.
The Landau Hamiltonian is 
\[
	H_B
	=
	\tildel_1^2 + \tildel_2^2.
\]
\end{definition}

Clearly, $H_B$ can be written in the form
\[
	H_B
	=
	(i \nabla - A)^2.
\]
with the \emph{magnetic potential} $A = \frac{B}{2} (- x_2, x_1)$.
Indeed, this is the so-called \emph{symmetric gauge} and any $A'$ with $\partial_1 A'_2 - \partial_2 A'_1 = B$ will lead to a unitarily equivalent operator.
It is well-known that $H_B$ is a self-adjoint operator in $L^2(\RR^2)$, an operator core being $C_0^\infty(\RR^2)$, with spectrum $\sigma(H_B) = \{B, 3 B, 5B, \dots \}$.
Our first main result is:

\begin{theorem}
	\label{thm:main}
	Let $B > 0$ and let $S \subseteq \RR^2$ be $(\ell, \rho)$-thick.
	Then, there are $C_1, C_2, C_3, C_4 > 0$, such that for all $E > 0$ we have
	\begin{equation*}
		\label{eq:spectral_inequality_main_theorem}
	\lVert f \rVert_{L^2(\RR^2)}^2
	\leq
	\left(
	\frac{C_1}{\rho}
	\right)^{C_2 + C_3 \lvert \ell \rvert_1 \sqrt{E} + C_4 (\lvert \ell \rvert_1^2 B )}
	\lVert f \rVert_{L^2(S)}^2
	\quad
	\text{for all $f \in \ran \mathbf{1}_{(-\infty, E]}(H_B)$}.
	\end{equation*}
\end{theorem}
Let us comment on the expression
\[
	\left(
	\frac{C_1}{\rho}
	\right)^{C_2 + C_3 \lvert \ell \rvert_1 \sqrt{E} + C_4 (\lvert \ell \rvert_1^2 B )}
\]	
\begin{enumerate}
	\item
	\textbf{In the limit $B \to 0$}, the constant converges to the expression for the pure Laplacian in the Logvinenko-Sereda-Kovrikijne theorem~\cite{Kovrijkine-00}.
	So, one can indeed also set $B = 0$ in the statement of Theorem~\ref{thm:main} and in this sense, the dependence $\exp(C \sqrt{E})$ is optimal.
	Indeed, this is the first time that we aware of any dependence on $E$ in a spectral inequality for the Landau operator, and this precise dependence is necessary for the application to controllability of the magnetic heat equation, in Section~\ref{sec:heat}.
	\item
	\textbf{The relation of $E$ to $\ell$, and $B$ to $\ell$ is optimal}.
	Since $H_B$ is of second order in $\partial_1$, $\partial_2$ and of the same order in $B$, simultaneous scaling in $E$ and in $B$ corresponds to the square of the inverse scaling in space.
	\item
	\textbf{The term $\lvert \ell \rvert_1^2 B$ in the exponent is optimal} when $\lvert \ell \rvert_1$ is sent to $\infty$, see Remark~\ref{rem:large_ell_optimal}.
	\item
	\textbf{The dependence on $\lvert \ell \rvert_1$} yields a meaningful limit in the \textbf{homogenization regime}, that is when $\ell \to 0$:
	In this regime, the maximal size of holes in the set $S \subseteq \RR^2$ becomes small.
	On the one hand, since $\sqrt{E} \lvert \ell \rvert_1 \gg B \lvert \ell \rvert_1^2$ as $\ell \to 0$, we observe that in the homogenization regime, the influence of the magnetic field $B$ in the spectral inequality (at fixed $B$ and $E \geq B$) fades.
	On the other hand, when sending $\ell \to 0$, the exponent will disappear and the observation operator $\mathbf{1}_S$ strongly converges to an $E$-independent operator, see also the discussion in~\cite{NakicTTV-20}.
	\item
	\textbf{Thickness of $S$ is necessary} for any quantitative unique continuation principle of the form
	\[
	\lVert f \rVert_{L^2(\RR^2)}^2
	\leq
	C(E,\ell,B)
	\lVert f \rVert_{L^2(S)}^2
	\quad
	\text{for all $f \in \ran \mathbf{1}_{(-\infty, E]}(H_B)$}.
	\]
	This is proved in Theorem~\ref{thm:thickness_necessary_for_UCP}.
\end{enumerate}

We also have the corresponding result for finite-volume restrictions $H_{B,L}$ of $H_B$ onto boxes $\Lambda_L = (0, L_1) \times (0, L_2) \subseteq \RR^2$, where $L = (L_1,L_2) \in \RR_{> 0}^2$ satisfies the so-called~\emph{integer flux condition}, and $H_L$ is defined with appropriate magnetic boundary conditions, see Section~\ref{sec:finite_volume} for precise definitions.

\begin{theorem}
		\label{thm:main_bounded_domain}
		Let $B > 0$, and let $S \subseteq \RR^2$ be $(\ell, \rho)$-thick.
		Then there are $C_1, C_2, C_3, C_4 > 0$, such that for all
		$L = (L_1,L_2) \in (0, \infty)^2$ satisfying the integer flux condition
		\[
		B (L_2 - L_1) \in 2 \pi \ZZ,
		\]
		and
		$\ell_1 \leq L_1$, $\ell_2 \leq L_2$,
		we have
		\[
		\lVert f \rVert_{L^2(\Lambda_L)}^2
		\leq
		\left(
		\frac{C_1}{\rho}
		\right)^{C_2 + C_3 \lvert \ell \rvert_1 \sqrt{E} + C_4 (\lvert \ell \rvert_1^2 B )}
		\lVert f \rVert_{L^2(\Lambda_L \cap S)}^2
		\quad
		\text{for all $f \in \ran \mathbf{1}_{(-\infty, E]}(H_{B,L})$}.
	\]
\end{theorem}
	
Estimates as in Theorem~\ref{thm:main_bounded_domain} have been commonly used in the context of the spectral theory of random Schr\"odinger operators where they are also known as \emph{Quantitative Unique continuation principles}, see \cite{CombesHK-03, CombesHKR-04, GerminetKS-07, CombesHK-07, TaeuferV-16}.

However, previous results neither carried the explicit dependence on the parameters $E,B$ nor were they valid beyond open sets, whereas Theorem~\ref{thm:main_bounded_domain} allows for any subset $S \subset \Lambda_L$ of positive measure.
 We explain in Section~\ref{sec:RSO} how this leads to improvements.

\section{Magnetic Bernstein inequalities}
	\label{sec:Magnetic_Bernstein_inequalities}

In this section, we state and prove magnetic Bernstein inequalities.
The first step will be to express, for $f \in \ran \mathbf{1}_{(- \infty, E]}(H_B)$ and $m \in \NN$, sums of derivatives of the form
\[
	\sum_{\alpha \in \{1,2\}^m}
	\lVert 
		\tildel_{\alpha_1} \tildel_{\alpha_2} \dots \tildel_{\alpha_m} 
		f
	\rVert_{L^2(\RR^2)}^2,
\]
where $(\alpha_1, \dots, \alpha_m)$ are the entries of a vector $\alpha \in \{1,2\}^m$, in terms of $H_B$.
For this purpose, we need to better understand the algebra generated by the magnetic derivatives $\tildel_1, \tildel_2$.

The magnetic derivatives $\tildel_1$ and $\tildel_2$ satisfy the commutator relation
	\begin{align}
	\label{eq:commutator_relation}
		 [\tildel_1 ,  \tildel_2 ] &= \left[i \partial_1 - \frac{B}{2} x_2 , i \partial_2 + \frac{B}{2} x_1 \right] 
		 =
		 \left[i \partial_1, \frac{B}{2} x_1 \right] - \left[\frac{B}{2} x_2, i\partial_2 \right] = iB.
	\end{align}
Consider the algebra $\mathcal{X}$ of all polynomials in $\tildel_1, \tildel_2$ modulo the commutator relation~\eqref{eq:commutator_relation}.
Clearly, polynomials of $H_B$ form a subalgebra of $\mathcal{X}$.
We define a linear operator $R$ mapping $\mathcal{X}$ to itself, defined via
\[
	R(P)
	=
	\tildel_1 P \tildel_1
	+
	\tildel_2 P \tildel_2
\]
for any polynomial $P$ in the variables $\tildel_1$ and $\tildel_2$.
The key idea is now that
\begin{align}	
	\label{eq:explain_R}
	\sum_{\alpha \in \{1,2\}^m}
	\lVert 
		\tildel_{\alpha_1} \tildel_{\alpha_2} \dots \tildel_{\alpha_m} f
	\rVert_{L^2(\RR^2)}^2
	=
	\left \langle f , R^m (\operatorname{Id}) f \right \rangle
\end{align}
for sufficiently regular $f$, say $f \in \mathcal{S}(\RR^2)$.
This can be verified by integration by parts and is explained in the proof of Theorem~\ref{thm:magnetic_Bernstein}.

The following Lemma~\ref{lem:recursion} is the first key result of this section, stating that $R^m(\operatorname{Id})$ is not only a polynomial in the variables $\tildel_1, \tildel_2$, but actually a polynomial in $H_B$.
The subsequent Lemma~\ref{lem:bound_on_F_n} then provides an explicit bound on this polynomial, allowing to replace $R^m(\operatorname{Id})$ in~\eqref{eq:explain_R} by a polynomial in $H_B$.

\begin{lemma}
	\label{lem:recursion}
	For all $m \geq 0$, the operator given by $R^m(\operatorname{Id})$ is a polynomial in $H_B$, which we denote by $F_m$, that is
	\[
	F_m(H_B) := R^m(\operatorname{Id}).
	\]
	Furthermore,
	\begin{equation}
	\label{eq:recursion}
	F_{m+1}(H_B)
	=
	R(F_m(H_B))
	=
	\frac{1}{2}
	\left( \left( H_B - B \right) F_m( H_B - 2B) + \left( H_B + B \right) F_m (H_B + 2 B) \right).
	\end{equation}
\end{lemma}

\begin{proof}
	Since $R$ is linear, it suffices to consider monomials, and to see~\eqref{eq:recursion}, it certainly suffices to show
	\begin{equation*}
	\label{eq:recursion_R}
	2 R(H_B^n)
	=
	\left( H_B - B \right) (H_B - 2 B)^n 
	+ 
	\left( H_B + B \right) (H_B + 2 B)^n
	\end{equation*}
	for each $n \geq 0$.  We have	
	\begin{align*}
	R(H_B^n) 
	&=
	\tildel_1 H_B^n \tildel_1
	+
	\tildel_2 H_B^n \tildel_2.
	\end{align*}
	Define
	\[
	X_n := i \tildel_2 H_B^{n-1} \tildel_1 - i \tildel_1 H_B^{n-1} \tildel_2,
	\quad
	\text{and}
	\quad
	Y_n := \tildel_1 H_B^{n-1} \tildel_1 + \tildel_2 H_B^{n-1} \tildel_2.
	\]
	The commutator identity~\eqref{eq:commutator_relation} leads to	
	\[
		H_B \tildel_1 = \tildel_1 H_B - 2 i B \tildel_2,
		\quad
		\text{and}
		\quad
		H_B \tildel_2 = \tildel_2 H_B + 2 i B \tildel_1.
	\]
	In particular, this implies $X_1 = B$, $Y_1 = H_B$, as well as
	\begin{align*}
	\begin{pmatrix}
	X_{n+1}
	\\
	Y_{n+1}
	\end{pmatrix}
	=
	\begin{pmatrix}
	H_B & 2 B \\
	2 B & H_B \\
	\end{pmatrix}
	\begin{pmatrix}
	X_{n}
	\\
	Y_{n}
	\end{pmatrix}
	\end{align*}
	Diagonalizing 
	\[
	\begin{pmatrix}
	H_B & 2 B \\
	2 B & H_B \\
	\end{pmatrix}
	=
	\frac{1}{2}
	\begin{pmatrix}
	-1 & 1 \\
	1  & 1 \\
	\end{pmatrix}
	\begin{pmatrix}
	H_B - 2 B & 0 \\
	0 & H_B + 2 B \\
	\end{pmatrix}
\begin{pmatrix}
	-1 & 1 \\
	1  & 1 \\
	\end{pmatrix}
	\]
	we obtain
	\[
	\begin{pmatrix}
	X_{n+1}
	\\
	Y_{n+1}
	\end{pmatrix}
	=
	\begin{pmatrix}
	H_B & 2 B \\
	2 B & H_B \\
	\end{pmatrix}^n
	\begin{pmatrix}
	B
	\\
	H_B
	\end{pmatrix}
	=
	\frac{1}{2} 
	\begin{pmatrix}
	-1 & 1 \\
	1  & 1 \\
	\end{pmatrix}
	\begin{pmatrix}
	H_B - 2B & 0 \\
	0 & H_B +2 B \\
	\end{pmatrix}^n
\begin{pmatrix}
	-1 & 1 \\
	1  & 1 \\
	\end{pmatrix}
	\begin{pmatrix}
	B
	\\
	H_B
	\end{pmatrix}
	\]
	which leads to
	\[
	2 R(H_B^n)
	=
	2 Y_{n+1}
	=
	\left( H_B - B \right) (H_B - 2 B)^n 
	+ 
	\left( H_B + B \right) (H_B + 2 B)^n.
	\qedhere
	\]

\end{proof}

\begin{remark}
	It seems that Lemma~\ref{lem:recursion} relies on the particular structure of the two-dimensional Landau Hamiltonian.
	Indeed, the recursive technique already breaks down in the case of three-dimensional constant magnetic fields: Given a magnetic field $0 \neq (B_1, B_2, B_3) \in \RR^d$, the corresponding magnetic Schrödinger operator can be written in the form
	\[
	H_{(B_1, B_2, B_3)}
	=
	\tildel_1^2 + \tildel_2^2 + \tildel_3^2
	\]
	where the covariant derivatives $\tildel_1$, $\tildel_2$, $\tildel_3$ must satisfy the commutator relations
	\[
	[\tildel_1, \tildel_2] = i B_3,
	\quad
	[\tildel_2, \tildel_3] = i B_1,
	\quad
	[\tildel_3, \tildel_1] = i B_2.
	\]
	Now, if one defines the three-dimensional version $R_{(3)}$ of the operator $R$, acting as
	\[
	R_{(3)}(P)
	=
	\tildel_1 P \tildel_1
	+
	\tildel_2 P \tildel_2
	+
	\tildel_3 P \tildel_3
\]
	on polynomials in $\tildel_1, \tildel_2, \tildel_3$, it is straightforward to calculate that already $R^2(\operatorname{Id})$ will no longer be a polynomial in $H_{(B_1, B_2, B_3)}$.
	We conclude that in order to treat operators beyond the two-dimensional Landau operators -- in particular higher-dimensional magnetic Schrödinger operators based on covariant derivatives with a richer commutator algebra -- one might need to use arguments beyond the ones used in the proof of Lemma~\ref{lem:recursion}.
\end{remark}

Next, we use the recursive identity~\eqref{eq:recursion} to provide explicit bounds on the $F_m$.

\begin{lemma}
	\label{lem:bound_on_F_n}
	For every $t \in B (2 \NN + 1)$ and $m \in \NN$ we have
	\[
	\frac{1}{2^m}
	(t + B) (t + 3 B) \dots (t + (2m-1) B)
	\leq
	F_m(t)
	\leq
	(t + B) (t + 3 B) \dots (t + (2m-1) B).
	\]
	In particular,
	\[
	\lVert F_m (H_B) \mathbf{1}_{(- \infty, E]}(H_B) \rVert
	=
	\max 
	\left\{
		\lvert F_m(t) \rvert
		\colon
		t \in \sigma(H_B) \cap (- \infty, E]
	\right\}
	\leq
	(E + m B)^m.
	\]
\end{lemma}

\begin{proof}
	By an iterative application of Lemma~\ref{lem:recursion}, $F_n(t)$ can be expressed as $2^{-n}$ times a sum of $2^n$ many products of factors of the form $(t - k B)$. 
	Each summand must have a term $(t \pm B)$, and parameters $k$ in neighbouring factors differ by $-2, 0$ or $+2$.
	Furthermore, as soon as $t - k B$ is zero, the summand containing this factor will vanish whence each summand is non-negative.
	The lower bound follows by dropping all but one term.
	The upper bound follows by replacing all $2^n$ many summands by the expression that maximises such products.
\end{proof}

With this, we can prove the magnetic Bernstein inequalities:
\begin{theorem}
	\label{thm:magnetic_Bernstein}
	For every $E, B \geq 0$ and $m \in \NN$, we have the magnetic Bernstein inequality
	\begin{equation}
	\label{eq:magnetic_Bernstein}
	\sum_{\alpha \in \{1,2\}^m}
	\lVert 
		\tildel_{\alpha_1} \tildel_{\alpha_2} \dots \tildel_{\alpha_m} f
	\rVert_{L^2(\RR^2)}^2
	\leq
	C_B(m)
	\lVert f \rVert_{L^2(\RR^2)}^2
	\quad
	\text{for all}
	\quad 
	f \in \operatorname{Ran} \mathbf{1}_{(- \infty, E]} (H_B),
	\end{equation}
	where
	\[
	C_B(m)
	=
	(E + B m)^m.
	\]
\end{theorem} 

\begin{proof}
Note that $\mathbf{1}_{(- \infty,E]}(H_B)$ is a finite sum of projectors onto the Landau Levels up to $E$.
These projectors have the integral kernel
\[
	K_{E,B}(x,y)
	=
	\frac{B}{2 \pi}
	\sum_{k \in \NN_0 \colon (2 k + 1) B \leq E}
	\exp
	\left(
		- \frac{B}{4} \lvert x - y \rvert^2 - i \frac{B}{2} (x_1 y_2 - x_2 y_1)
	\right)
	\mathcal{L}_k
	\left(
	\frac{B}{2}
	\lvert x - y \rvert^2
	\right)
\]
for $x,y \in \RR^2$ where $x_1,x_2$ and $y_1, y_2$ denote the corresponding first and second entries of $x$ and $y$, respectively, and the $\mathcal{L}_k$ are the Legendre polynomials, see~\cite{Fock-1928, Landau-1930}.
This kernel is smooth, exponentially decaying and therefore leaves the Schwarz space $\mathcal{S}(\RR^2)$ invariant, that is
\[
\mathbf{1}_{(- \infty, E]}(H_B) f \in \mathcal{S}(\RR^2)
\quad
\text{for all $f \in \mathcal{S}(\RR^2)$}.
\]
This allows to use integration by parts for the magnetic derivatives, and we calculate
\begin{align*}
	&\left\langle 
	\mathbf{1}_{(- \infty, E]}(H_B) f,
	F_n (H_B)
	\mathbf{1}_{(- \infty, E]}(H_B) f
	\right\rangle 
	=
	\left\langle 
	\mathbf{1}_{(- \infty, E]}(H_B) f,
	R^n(\operatorname{Id})
	\mathbf{1}_{(- \infty, E]}(H_B) f
	\right\rangle
	\\
	=
	&\sum_{\alpha \in \{1,2\}^m}
	\left\langle 
	\mathbf{1}_{(- \infty, E]}(H_B) f,
		\tildel_{\alpha_m} \tildel_{\alpha_{m-1}} \dots \tildel_{\alpha_1}
		\tildel_{\alpha_1} \tildel_{\alpha_2} \dots \tildel_{\alpha_m} 
		\mathbf{1}_{(- \infty, E]}(H_B) f
	\right\rangle 
	\\
	=
	&\sum_{\alpha \in \{1,2\}^m}
	\lVert 
		\tildel_{\alpha_1} \tildel_{\alpha_2} \dots \tildel_{\alpha_m} \mathbf{1}_{(- \infty, E]}(H_B) f
	\rVert_{L^2(\RR^2)}^2
\end{align*}
for all $f \in \mathcal{S}(\RR^2)$.
By density, this extends to all $f \in L^2(\RR^2)$. 
Together with Lemma~\ref{lem:bound_on_F_n}, we obtain the claim.
\end{proof}

\begin{remark}
	The classic Bernstein inequalities (in two dimensions) are
	\[
	\sum_{\alpha \in \{1,2\}^m}
	\lVert
	\partial_{\alpha_1}
	\partial_{\alpha_2}
	\dots
	\partial_{\alpha_m}
	f
	\rVert_{L^2(\RR^2)}^2
	\leq
	E^m
	\lVert f \rVert_{L^2(\RR^2)}^2
	\quad
	\text{for all}
	\quad
	f \in \ran \mathbf{1}_{(- \infty,E]} (- \Delta).
	\]
	They are an immediate consequence of the identity
	\[
	\sum_{\alpha \in \{1,2\}^m}
	\lVert
	\partial_{\alpha_1}
	\partial_{\alpha_2}
	\dots
	\partial_{\alpha_m}
	f
	\rVert_{L^2(\RR^2)}^2
	=
	\left\langle f, (- \Delta)^m f \right\rangle
	\]
	for sufficiently regular $f$, which follows from a repeated application of integration by parts.
	Note that, in contrast to the magnetic derivatives $\tildel_1, \tildel_2$, the classic derivatives $\partial_1, \partial_2$ commute. 
	Hence, one usually writes the right hand side of classic Bernstein inequalities in multi-index notation in the equivalent form
	\[
	\sum_{\lvert \mathfrak{n} \rvert = m}
	\frac{1}{\mathfrak{n}!}	
	\lVert 
	\partial^{\mathfrak{n}}
	f	
	\rVert_{L^2(\RR^2)}^2
	\leq
	\frac{E^m}{m!}
	\lVert 
	f	
	\rVert_{L^2(\RR^2)}^2,
	\]
	see~\cite{EgidiS-21} for an overview.
	For other operators, Bernstein-type inequalities are rather rare. 
	One notable exception where Bernstein-type estimates are known is the Harmonic Oscillator~\cite{BeauchardJPS-21, EgidiS-21}.

\end{remark}

From the proof of Theorem~\ref{thm:magnetic_Bernstein} it also follows that $\operatorname{Ran} \mathbf{1}_{(- \infty, \mu]} (H_B) \subseteq C^\infty(\RR^2)$. 

\begin{remark}
	\label{rem:bound_impossible}
It is paramount to work with \emph{magnetic derivatives} $\tildel_1, \tildel_2$ in Theorem~\ref{thm:magnetic_Bernstein}, and not with \emph{ordinary derivatives} $\partial_1, \partial_2$.
Indeed, derivatives of $f \in \operatorname{Ran} \mathbf{1}_{(- \infty, E]}(H_B)$ will \emph{not} be uniformly bounded in $L^2(\RR^2)$ for fixed $E$ and $B$. 
We illustrate this with the following example: 
Let $0 < B \leq E$ and consider, for $y \in \mathbb R^2$, the eigenfunction to the eigenvalue $B$ 
\begin{equation}
	\label{eq:definition_f_y}
	f_y(x) 
	:= 
	\exp
	\left( 
		- \frac{B}{4} \lvert x-y \rvert^2 
		- i \frac{B}{2} (x_1y_2 - x_2y_1) 
	\right)
	\in
	\operatorname{Ran} \mathbf{1}_{(- \infty,E]}(H_B).
\end{equation}
Clearly, $\lVert f_y \rVert_{L^2(\RR^2)}^2 = \frac{2 \pi }{B}$ is independent of $y$. 
However,
\begin{align*}
    \lVert \partial_1 f_y \rVert_{L^2(\mathbb R^2)}^2
    =&
    \frac{B}{2}
	\int_{\RR^2}
	\lvert
	( - (x_1 - y_1) - i y_2)
	f_y(x)
	\rvert^2
	\drm x
	\geq
	\frac{B}{2}
	\int_{\RR^2}
	\left(	
		\lvert y_2 \rvert^2
		-
		\lvert x_1 - y_1 \rvert^2
	\right)
	\lvert
	f_y(x)
	\rvert^2
	\drm x
	\\
	&=
	\frac{\pi \lvert y_2 \rvert^2}{2}
	-
	\frac{B}{2}
	\int_{\RR^2}
	x_1^2	
	\exp
	\left(
		- \frac{B \lvert x \rvert^2}{4}
	\right)
	\drm x
	=
	\frac{\pi \lvert y_2 \rvert^2}{2} - \frac{4 \pi}{B}.
\end{align*}
This can be made arbitrarily large by choosing $\lvert y_2 \rvert$ sufficiently large. 
Consequently, Theorem~\ref{thm:magnetic_Bernstein} cannot hold verbatim when replacing $\tildel_1, \tildel_2$ by $\partial_1, \partial_2$.
\end{remark}

\begin{remark} 
	\label{rem:large_ell_optimal}
Choosing $y = 0$ in~\eqref{eq:definition_f_y}, the function $f_0$ also demonstrates that for fixed $E, \rho > 0$ the constant
\[
	\left(
	\frac{C_1}{\rho}
	\right)^{C_2 + C_3 \lvert \ell \rvert_1 \sqrt{E} + C_4 (\lvert \ell \rvert_1^2 B )}
	\sim
	\tilde C_1
	\exp
	\left(
		\tilde C_2 + \tilde C_3 \lvert \ell \rvert_1 \sqrt{E} + \tilde C_4 (\lvert \ell \rvert_1^2 B )
	\right)
\]
in Theorem~\ref{thm:main} has the optimal behavior as $\lvert \ell \rvert_1^2$ tends to $\infty$.
Indeed, let $\ell = (\ell_1, \ell_2) \in (0, \infty)^2$ and consider the $(\ell, \rho)$-thick set 
\[	
	S 
	:= 
	\{
	x \in \RR^2 \colon \lvert x \rvert \geq r
	\}
	\quad
	\text{where $r:=\max( \ell_1,\ell_2) (1-\rho)/2$}.
\] 
Then,
\begin{align*}
	\lVert f_0 \rVert_{L^2(S)}^2
	&= 
	\int_r^{\infty} 2 \pi s  \exp(-B s^2/2) \mathrm d s 
	= 
	\frac {2 \pi} {B} \exp(- B r^2/2)
	\\
	&=
	\exp
	\left(
		\frac{- B \max(\ell_1 ,\ell_2)^2 (1 - \rho)^2}{8}
	\right)
	\lVert f_0 \rVert_{L^2(\RR^d)}^2
	\le
	\exp
	\left(
		\frac{- B  \lvert \ell \rvert_1^2 (1 - \rho)^2}{2}
	\right)
	\lVert f_0 \rVert_{L^2(\RR^d)}^2 \, .
\end{align*}
Thus, the constant in Theorem~\ref{thm:main} must at least be of order $\exp(C B \lvert \ell \rvert_1^2)$ as $\lvert \ell \rvert_1 \to \infty$. 
\end{remark}

We can furthermore use the functions $f_y$, defined in~\eqref{eq:definition_f_y}, to show that thickness is necessary for any quantitative unique continuation principle on spectral subspaces.

\begin{theorem}
	\label{thm:thickness_necessary_for_UCP}
	Assume that $S \subset \RR^2$ is such that for some $B > 0$, $E \geq B$, there is a constant $C > 0$ such that 
	\begin{equation}
	\label{eq:abstract_UCP}
	\lVert f \rVert_{L^2(\RR^2)}^2
	\leq
	C
	\lVert f \rVert_{L^2(S)}^2
	\quad
	\text{for all $f \in \ran \mathbf{1}_{(-\infty, E]}(H_B)$}.
	\end{equation}
	Then $S$ is thick.
\end{theorem}

\begin{proof}
	If $S \subset \RR^2$ was not thick, there would be $\{ y^{(n)} \}_{n \in \NN} \subset \RR^2$ with
	\[
	\vol ( B_n(y^{(n)}) \cap S) \leq \frac{1}{n}
	\quad
	\text{for all $n \in \NN$}
	\]
	where $B_r(x)$ denotes the open ball of radius $r > 0$ around $x \in \RR^2$.
	Defining $f_{y^{(n)}}$ as in~\eqref{eq:definition_f_y}, we have $f_{y^{(n)}} \in \ran \mathbf{1}_{( - \infty, E]}(H_B)$ with $\lVert f_{y^{(n)}} \rVert_{L^2(\RR^2)}^2 = \frac{2 \pi}{B}$, but
	\begin{align*}
	\lVert f_{y^{(n)}} \rVert_{L^2(S)}^2
	&\leq
	\lVert f_{y^{(n)}} \rVert_{L^\infty(\RR^2)}^2
	\cdot
	\vol ( B_n(y^{(n)}) \cap S)
	+
	\int_{\lvert x - y_n \rvert \geq n}
	\exp
	\left( - \frac{B}{2} \lvert x - y^{(n)} \rvert \right)
	\drm x
	\\	
	&\leq
	\frac{1}{n}
	+
	\frac{1}{B} \exp \left( - \frac{B n^2}{2} \right).
	\end{align*}
	This tends to $0$ as $n \to \infty$, so~\eqref{eq:abstract_UCP} cannot hold. 
\end{proof}

In the classic strategy of proof of the Kovrijkine-Logvinenko-Sereda theorems, we would now like to bound the $L^2(\RR^2)$-norm of higher order \emph{ordinary derivatives} $\partial_1, \partial_2$ of $f \in \operatorname{Ran} \mathbf{1}_{(- \infty, E]} (H_B)$ and use this to infer that $f$ is analytic.
Unfortunately, in light of Remark~\ref{rem:bound_impossible}, this is impossible.
However, a closer look shows that this lack of a uniform bound is due to an oscillating phase factor for large $\lvert x \rvert$.
This suggests that, instead of proving $L^2(\RR^d)$-bounds on (derivatives of) $f$, we might be better off proving $L^1(\RR^2)$-bounds on derivatives of $\lvert f \rvert^2$.
For this, we need some notation. 
For a finite sequence $\alpha = (\alpha_1, \dots, \alpha_m) \in \{1, 2 \} ^m$, let
\begin{align*}
	\partial^\alpha :=& \partial_{\alpha_1} \partial_{\alpha_2} \dots \partial_{\alpha_m}, \\
	\tildel^\alpha := &\tildel_{\alpha_1} \tildel_{\alpha_2} \dots \tildel_{\alpha_m}.
\end{align*}
Furthermore, we write $\beta \le \alpha$, if $\beta$ is a subsequence of $\alpha$ and write $\alpha \setminus \beta$ for the complementary subsequence.
We can now formulate our next theorem which are Bernstein-type inequalities (with ordinary derivatives) on $\lvert f \rvert^2$ where $f \in \operatorname{Ran} \mathbf{1}_{(- \infty, E]}(H_B)$:

\begin{theorem}
	\label{thm:magnetic_Bernstein_2}
For every $E, B \geq 0$, $m \in \NN$, and $f \in \operatorname{Ran} \mathbf{1}_{(- \infty, E]} (H_B)$ we have 
	\begin{equation}
	\label{eq:magnetic_Bernstein_1}
	\sum_{\alpha \in \{1,2\}^m}
	\lVert 
		\partial^\alpha  \lvert f \rvert^2
 	\rVert_{L^1(\RR^2)}
	\leq
	C'_B(m)
	\lVert f  \rVert_{L^2(\RR^2)}^2 
	\quad
	\text{where}
	\quad
	C'_B(m)
	=
	2^{3m/2}(E + B m)^{m/2}.
	\end{equation}
	Furthermore, for all $\alpha \in \{0,1\}^m$
	\begin{equation}
	\label{eq:magnetic_Bernstein_2}
	\sum_{\alpha \in \{1,2\}^m}
	\lVert 
		\partial^\alpha 
		\lvert
		f
		\rvert^2
	\rVert_{L^\infty(\RR^2)}
	\leq
	\Csob
	\sum_{m' = m}^{m + 3}
	C'_B(m)
	\lVert f  \rVert_{L^2(\RR^2)}^2 
	\end{equation}
	where $\Csob > 0$ is a universal constant.
\end{theorem}

As the notation suggests, $\Csob$ comes from a Sobolev embedding.

\begin{proof}
Let $u,v \in \mathcal C^\infty (\mathbb R^2, \mathbb C) $ and $x \in \mathbb R^2$. We have
	\begin{align*}
		 i\partial_1 (u \bar v) (x)&= \bar v (x) \left(\left(i \partial_1 - \frac B 2 x_2 \right) u\right) (x) - u (x)  \overline{\left(\left(i \partial_1 - \frac B 2 x_2 \right) v\right)} (x) \\
		 &= \bar v(x) \tildel_1 u (x)  - u(x) \overline{ \tildel_1v}(x) .
	 \end{align*}
Analogously,
	\[i \partial_2 (u \bar v)(x)=\bar v(x) \tildel_2 (x)  - u(x) \overline{\tildel_2v}(x) .\]
By induction, for any $ \alpha \in \{1,2\}^n$, this leads to
	\[i^m \partial^\alpha \lvert u \rvert ^2 (x) = \sum_{\beta \le \alpha } (-1)^{m- \lvert \beta \rvert}  \tildel^\beta u(x)  \overline{  \tildel^{\alpha \setminus \beta} u(x)} .\]
Thus, we can estimate
	\begin{align*}
		\sum_{\lvert \alpha \rvert =m} \lVert \partial^\alpha \lvert f \rvert^2 \rVert_{L^1(\RR^2)} 
		\le & \sum_{\lvert \alpha \rvert=m} \sum_{\beta \le \alpha} \lVert \tildel^\beta f \rVert_{L^2(\RR^2)} \lVert \tildel^{\alpha\setminus \beta}  f  \rVert_{L^2(\RR^2)} \\
		=& \sum_{k=0} ^m \binom{m}{k} \sum_{\lvert \beta \rvert=k , \lvert \beta' \rvert=m-k}  
		\lVert \tildel^\beta  f  \rVert_{L^2(\RR^2)} 
		\lVert \tildel^{\beta'}  f  \rVert_{L^2(\RR^2)} 	\\
		\le & \sum_{k=0} ^m \binom{m}{k} 2^{m/2}  
		\sqrt{ \sum_{\lvert \beta \rvert=k , \lvert \beta' \rvert=m-k}  \lVert \tildel^\beta  f \rVert_{L^2(\RR^2)}^2 \lVert \tildel^{\beta'}f \rVert_{L^2(\RR^2)} ^2 } \\
		\le &  \sum_{k=0} ^m \binom{m}{k} 2^{m/2} \sqrt{ C_B(k) C_B(m-k)} \lVert f \rVert_{L^2(\RR^2)}^2  \\
		\le & \sum_{k=0} ^m \binom{m}{k} 2^{m/2} (E + Bm)^{m/2} \lVert f \rVert_{L^2(\RR^2)}^2= 2^{3m/2} (E + Bm)^{m/2} \lVert f \rVert_{L^2(\RR^2)}^2.
	\end{align*}
	Estimate~\eqref{eq:magnetic_Bernstein_2} follows from~\eqref{eq:magnetic_Bernstein_1} by using the Sobolev estimate $\lVert g \rVert_{L^\infty(\RR^2)} \leq \Csob \lVert g \rVert_{W^{3,1}(\RR^2)}$ which leads to
	\begin{align*}
	\sum_{\alpha \in \{1,2\}^m}
	\lVert \partial^\alpha \lvert f \rvert^2 \rVert_{L^\infty(\RR^2)}
	&\leq
	\Csob
	\sum_{\alpha \in \{1,2\}^m}
	\lVert \partial^\alpha \lvert f \rvert^2 \rVert_{W^{3,1}(\RR^2)}
	=
	\Csob
	\sum_{\alpha \in \{1,2\}^m}
	\sum_{\lvert \beta \rvert \leq 3}
	\lVert \partial^\beta \partial^\alpha \lvert f \rvert^2 \rVert_{L^1(\RR^2)}
	\\
	&=
	\Csob
	\sum_{m' = m}^{m+3}
	\sum_{\lvert \alpha' \rvert = m'}
	\lVert \partial^{\alpha'} \lvert f \rvert^2 \rVert_{L^1(\RR^2)}
	\leq
	\Csob
	\sum_{m' = m}^{m+3}
	C_B(m')	
	\lVert f \rVert_{L^2(\RR^2)}^2. 
	\qedhere
	\end{align*} 
	
\end{proof}

Let us emphasize that the the constant $\Csob$ comes from a Sobolev estimate in $\RR^2$.
When proving the analogous result on \emph{bounded} domains $\Lambda_L$ in Section~\ref{sec:finite_volume}, it is therefore desirable to work with restrictions onto \emph{one domain} (or shifted variants thereof).
This will be achieved by possibly extending functions beyond their original domain -- using the magnetic boundary conditions defined in Section~\ref{sec:finite_volume}.

\section{Spectral inequality for the Landau operator}
	\label{sec:proofs}
In this section, we prove Theorem~\ref{thm:main}.
The strategy of proof roughly follows Kovrijkine's proof~\cite{Kovrijkine-00} (for the pure Laplacian) and the more general argument in~\cite{EgidiS-21}, the latter being however formulated in an $L^2(\RR^d)$ setting instead of the $L^1(\RR^d)$ setting used here. 

\subsection{Analyticity and local estimate}

	\begin{lemma}
	\label{lem:analytic}
	Let $E \geq 0$, $f \in \ran \mathbf{1}_{(- \infty, E]}(H_B)$.	Then $\lvert f \rvert^2$ is analytic, i.e. it can be expanded in an absolutely convergent power series around every $x_0 \in \RR^2$.
	In particular, it has an analytic extension $\Phi$ to $\CC^2$.
	\end{lemma}

	\begin{proof}	
	Let $m \in \NN$.
	By~\eqref{eq:magnetic_Bernstein_2}, we have
	\begin{align*}
	\sum_{\alpha \in \{1,2\}^m}
	\lVert \partial^\alpha \lvert f \rvert^2 \rVert_{L^\infty(\RR^2)}
	&\leq
	\Csob \sum_{m' = m}^{m + 3}
	2^{3m'/2}
	(E + B m')^{m/2}
	\lVert f \rVert_{L^2(\RR^2)}^2 
	\\
	&\leq
	4
	\Csob
	2^{3 (m+3)/2}
	(E + B (m + 3))^{\frac{m+3}{2}}
	\lVert f \rVert_{L^2(\RR^2)}^2.
	\end{align*} 
	Using
	\[
	(E + B (m+3))^{\frac{m+3}{2}}
	\lesssim
	\sqrt{m!}
	\]
	we see that for every point $x_0 \in \RR^2$, the series	
	\[
	\Phi(z)
	:=
	\sum_{k = 0}^\infty
	\sum_{\lvert \alpha \rvert = k}
	\frac{\partial^\alpha \lvert f \rvert^2}{k!}
	(x - x_0)^\alpha,
	\]
	where $(x - x_0)^\alpha$ for $\alpha = (\alpha_1, \dots, \alpha_k) \in \{1,2\}^k$, means
	\[
	(x - x_0)^\alpha
	=
	(x - x_0)_1^{\alpha_1} \cdot (x - x_0)_2^{\alpha_2},
	\]
	converges absolutely and locally uniformly, agrees with $f$ on $\RR^2$, and defines an analytic extension $\Phi$ of $f$ to $\CC^2$.
	\end{proof}
	
	Next, we need a local lower bound on such analytic functions.
	For this purpose, given $r > 0$, we denote by $D_r = \{ z \in \CC \colon \lvert z \rvert \leq r \}$ the complex disc of radius $r$.
	We also need notation for two-dimensional complex polydiscs and, given $r_1, r_2 > 0$, we denote by $D_{(r_1,r_2)} = \{ (z_1, z_2) \in \CC^2 \colon \lvert z_j \rvert \leq r_j, j = 1,2 \}$ the complex polydisc with radii $r_1$ and $r_2$.

	\begin{lemma}
		\label{lem:local_estimate}
		Let $Q \subseteq \RR^2$ be an open rectangle with sides of lengths $\ell_1, \ell_2 > 0$, parallel to the coordinate axes, and let $g \colon Q \to \CC$ be a non-vanishing function admitting an analytic continuation $G$ to $Q + D_{(4 \ell_1, 4 \ell_2)} \subseteq \CC^2$.
		Then, for any measurable $U \subseteq Q$ and every linear bijection $A \colon \RR^2 \to \RR^2$, we have
		\begin{align*}
			\lVert g \rVert_{L^1(Q \cap U)}
			&\geq
			\frac{1}{2}
			\left(
				\frac{\vol (A ( Q \cap U))}{48 \pi \diam (A(Q))^2}
			\right)^{2 \frac{\log M}{\log 2}}
			\frac{\vol ( Q \cap U )}{\vol (Q)}
			\lVert g \rVert_{L^1(Q)}
			\\
			&\geq
			\frac{1}{2}
			\left(
				\frac{\vol A( Q \cap U )}{48 \pi \diam (A(Q))^2}
			\right)^{2 \frac{\log M}{\log 2} + 1}
			\lVert g \rVert_{L^1(Q)}
			,
		\end{align*}
		where
		\[
		M
		:=
		\frac{\vol ( Q )}{\lVert g \rVert_{L^1(Q)}}
		\cdot
		\sup_{z \in Q + D_{(4 \ell_1, 4 \ell_2)}}
		\lvert
		G(z)
		\rvert
		\geq
		1.
		\]
	\end{lemma}

	Similar statements as in Lemma~\ref{lem:local_estimate} can be found in several places in the literature.
	The original idea seems to go back to~\cite{Kovrijkine-00}.
	The proof provided here is inspired by proof of Lemma~3.5 in~\cite{EgidiS-21}, where a corresponding statement for $L^2$-norms is proved and the linear bijection $A$ was introduced.
	The latter will be used in subsequent steps of the proof of Theorem~\ref{thm:main} in order to optimize the constants. 
	Indeed, without $A$, the \emph{excentricity} of rectangles with side lengths $(\ell_1, \ell_2)$, or more precisely, the ratio between their diameter and their volume, would enter.
	The bijection helps to make the constant independent of the shape such that only the expression $\lvert \ell \rvert_1$ enters the final statement.
	  
	The proof of Lemma~\ref{lem:local_estimate} relies on a dimension reduction argument and the following one-dimensional estimate, due to~\cite{Kovrijkine-00}, which itself relies on the Remez inequality for polynomials as well as Bleschke products.

	\begin{lemma}[{Cf.~\cite[Lemma~1]{Kovrijkine-00}}]
    \label{lem:one-dimensional}
        Let $\varphi \colon D_{4 + \epsilon} \to \CC$ for some $\epsilon > 0$ be an analytic function with $\lvert \varphi(0) \rvert \geq 1$.
        Let $E \subseteq [0,1]$ be measurable with positive measure.
        Then
        \[
         \sup_{t \in [0,1]}
         \lvert \varphi(t) \rvert
         \leq
         \left(
            \frac{12}{\vol ( E )}
         \right)^{2 \frac{\log M_\varphi}{\log 2}}
         \sup_{t \in E}
         \lvert \varphi (t) \rvert
        \]
        where $M_\varphi = \sup_{z \in D_4} \lvert \varphi(z) \rvert$.
	\end{lemma}

	For convenience of the reader, we provide a proof of Lemma~\ref{lem:one-dimensional} in Appendix~\ref{app:Remez}.

	\begin{proof}[Proof of Lemma~\ref{lem:local_estimate}]
	For all $C > 0$, we clearly have
	\begin{align*}
		\lVert g \rVert_{L^1(Q \cap U)}
		&\geq
		\lVert 
		\mathbf{1}_{\{ x \in Q \cap U \colon \lvert g (x) \rvert > C \lVert g \rVert_{L^1(Q)}\}}
		\cdot g
		\rVert_{L^1(Q)}
		\\
		&\geq
		C
		\lVert g \rVert_{L^1(Q)}
		\cdot
		\vol 
		\left(
			x \in Q \cap U \colon \lvert g (x) \rvert > C \lVert g \rVert_{L^1(Q)}
		\right).
	\end{align*} 
	Using this with 
	\[
	C 
	=
	\left(
		\frac{\vol(A(Q \cap U)}{48 \pi \diam (A(Q))^2}
	\right)^{2 \frac{\log M}{\log 2}}
	\cdot
	\frac{1}{\vol ( Q )},
	\]
	the first stated inequality follows if we prove
	\[
	\vol \left(
			x \in Q \cap U \colon \lvert g (x) \rvert > C \lVert g \rVert_{L^1(Q)}
	\right)
	\geq
	\frac{\vol (Q \cap U)}{2}
	\]
	which is certainly the case if
	\begin{equation}
	\label{eq:inequality_W}
	\vol ( W )
	\leq
	\frac{\vol (Q \cap U)}{2},
	\quad
	\text{where}
	\quad
		W
		:=
		\left\{
			x \in Q
			\colon
			\lvert g(x) \rvert \leq C \lVert g \rVert_{L^1(Q)}
		\right\},
	\end{equation}
	i.e., the set $W$ where $\lvert g \rvert$ is ``small'' has no more than half of the Lebesgue mass of $Q \cap U$.
	To see~\eqref{eq:inequality_W}, we may assume without loss of generality $W \neq \emptyset$.
	We will first show that there is a line segment $I = I(y_0,W,Q) \subset Q$ of the form
	\[
	I
	=
	\{
	y_0 + t \xi_0 \colon t \in [0,t_{\max}]
	\}
	\]
	where $y_0, \xi_0 \in \RR^2$ and $t_{\max} > 0$,
	such that
	\begin{equation}
	\label{eq:line_segment_I}
	\frac{\vol_1 ( I \cap W)}{\vol_1 ( I )}
	\geq
	\frac{\vol (A(W))}{2 \pi \diam(A(Q))^2}.
	\end{equation}
	Indeed, there is $y_0 \in Q$ with $\lvert g(y_0) \rvert \geq \frac{\lVert g \rVert_{L^1(Q)}}{\vol (Q)}$.
		Using spherical coordinates around $A(y_0)$,
	\[
	\vol (A(W)) 
	=
	\int_0^{2 \pi}
	\int_0^\infty
	s
	\cdot
	\mathbf{1}_{A(W)}
	\left(
		A(y_0) + s 
		\begin{pmatrix} 
		\cos(\theta)\\ \sin(\theta) 
		\end{pmatrix}
	\right)
	\drm s
	\
	\drm \theta
	\]
	whence there exists $\xi_0 \in \RR^2$ with $\lvert \xi_0 \rvert = 1$ such that
	\[
	\vol ( A(W) )
	\leq
	2 \pi
	\int_0^\infty
	s
	\cdot
	\mathbf{1}_{A(W)}
	\left(
	A(y_0) 
	+ 
	s \xi_0 
	\right)
	\drm s
	.
	\]
	Defining
	\[	
	\eta_0
	:= 
	\frac{A^{-1}(\xi_0)}{\lvert A^{-1} (\xi_0) \rvert}
	\]
	and denoting by $I \subset Q$ the line segment of maximal length within $\overline Q$, that is
	\[
	I
	=
	\{ y_0 + t c_{\max} \cdot \eta_0 \ \colon t \in [0, 1] \}
	\]
	where $c_{\max} > 0$ is the maximal parameter such that $I \subset Q$ we have
	\[
	\vol ( A(W) )
	\leq
	2 \pi
	\vol_1 ( A(I \cap W) )
	\vol_1 ( A(I) ).
	\] 
	Taking into account $\diam(A(Q)) \geq \vol_1(I)$, the line segment $I \subset \overline{Q}$ indeed satisfies~\eqref{eq:line_segment_I}.
	Since $Q$ is open, there is $\epsilon > 0$ such that $y_0 + \vol_1 (I) \cdot \eta_0 z \in Q + D_{(4 \ell_1, 4 \ell_2)}$ for $z \in D_{4 + \epsilon}$.
	Define
	\[
	\varphi (z)
	:=
	\frac{\vol (Q)}{\lVert g \rVert_{L^1(Q)}}
	\cdot 
	G(y_0 + \vol_1 (I) \cdot \eta_0\ z)
	\in
	\CC
	\]
	where we recall that $G$ denotes the analytic continuation of $g$.
	By assumption, $\varphi$ is analytic on $D(4 + \epsilon)  \subset \CC$, and satisfies
	   \[
     \sup_{t \in [0,1]} \lvert \varphi(t) \rvert
     \geq
     \lvert \varphi(0) \rvert
     =
     \frac{\vol (Q) \cdot \lvert g(y_0) \rvert}{\lVert g \rVert_{L^1(Q)}}
     \geq
     1
    \]
    as well as
    \[
     M
     :=
     \sup_{z \in D(4)}
     \lvert \varphi(z) \rvert
     \leq
     \frac{\vol (Q)}{\lVert g \rVert_{L^1(Q)}}
     \sup_{z \in y_0 + D_{(4 \ell_1, 4 \ell_2)}}
     \lvert G(z) \rvert
     .
    \]
    We may assume $M > 1$ because if $M = 1$, then $g$ would be constant on $Q$ and the statement would follow immediately.
    Applying Lemma~\ref{lem:one-dimensional} with $E := \{ t \in [0,1] \colon y_0 + \vol_1 (I) \cdot \eta_0 \ t \in I \cap W \} \subseteq [0,1]$ yields
    \[
     \sup_{t \in E}
     \lvert \varphi(t) \rvert
     \geq
     \left(
        \frac{\vol (E) }{12}
     \right)^{2 \frac{\log M}{\log 2}}
     \sup_{t \in [0,1]}
     \lvert
     \varphi(t)
     \rvert
     \geq
     \left(
        \frac{\vol (E)}{12}
     \right)^{2 \frac{\log M}{\log 2}}.
    \]
    Using the definition of $\varphi$ and recalling that $G \mid_Q = g$, this becomes
    \[
     \sup_{t \in E}
     \lvert g(y_0 + \vol_1 (I) \cdot \eta_0 t) \rvert
     \geq
     \left(
        \frac{\vol (E) }{12}     
    \right)^{2 \frac{\log M}{\log 2}}
    \cdot
    \frac{\lVert g \rVert_{L^1(Q)}}{\vol (Q)}.
    \]
	Since $\vol (E) = \frac{\vol_1 (I \cap W)}{\vol (I)} \geq
	\frac{\vol (A(W))}{2 \pi \diam(A(Q))^2}$ by~\eqref{eq:line_segment_I}, we infer
	 \[
     \left(
        \frac{\vol (A(W))}{24 \pi \diam (A(Q))^2}
     \right)^{2 \frac{\log M}{\log 2}}
     \cdot
     \frac{\lVert g \rVert_{L^1(Q)}}{\vol (Q)}
     \leq
     \sup_{x \in W}
     \lvert g(x) \rvert.
    \]
    Combining this with the definition of $W$, we obtain
    \begin{align*}
     \sup_{x \in W}
     \lvert g(x) \rvert
     &\leq
     \left(
        \frac{\vol (A(Q \cap U))}{48 \pi \diam (A(Q))^2}
     \right)^{2 \frac{\log M}{\log 2}}
     \cdot
     \frac{\lVert g \rVert_{L^1(U)}}{\vol (Q)}
     \\
     &=
     \left(
     	\frac{ \vol (A(Q \cap U))}{2 \vol (A(W))}
     	\cdot
     	\frac{ \vol (A(W))}{ 24 \pi \diam (A(Q))^2}
     \right)^{2 \frac{\log M}{\log 2}}
     \cdot
     \frac{\lVert g \rVert_{L^1(U)}}{\vol Q}
     \\
     &\leq
     \left(
     	\frac{\vol (Q \cap U)}{2 \vol (W)}
     \right)^{2 \frac{\log M}{\log 2}}
     \sup_{x \in W} \lvert g(x) \rvert.
    \end{align*}
	Recalling $M > 1$, this implies $\vol (Q \cap U) \geq 2 \vol (W)$ and concludes the proof of the first stated inequality.
	The second inequality follows from $\vol (A(Q)) \leq \pi \diam (A(Q))^2$.
\end{proof}

\subsection{Good and bad rectangles}
	We cover $\RR^2$ by a family $(Q_j)_{j \in \NN}$ of open rectangles of side lengths $\ell_1$ and $\ell_2$, parallel to the coordinate axes, such that any two rectangles do not overlap and the complement of their union is a measure zero set.
	
	\begin{remark}
	\label{rem:overlap}
	To prove Theorem~\ref{thm:main_bounded_domain}, the analogous result to Theorem~\ref{thm:main} on rectangles $\Lambda_L$, the side lengths $L_1, L_2$ of the domain might not be multiples of $\ell_1, \ell_2$ and we will not be able to cover $\Lambda_L$ perfectly by a union of small rectangles of side lengths $\ell_1, \ell_2$.
	However, we can obtain a covering such that every point is contained in at most four elements of the covering.
	So, the arguments below will have to be amended with a factor four, see also~\cite{EgidiS-21}, where this argument is elaborated in a more general setting, using a more general notion of coverings.
	\end{remark}

	\begin{definition}
		\label{def:good_bad}
		Given $E, B > 0$ and $f \in \ran \mathbf{1}_{(- \infty, E]}(H_B)$, we call a rectangle $Q_j$ \emph{good} if
		\[
		\lVert \partial^\alpha \lvert f \rvert^2 \rVert_{L^1(Q_j)}
		\leq
		4^{m+1}
		C_B'(m) 
		\lVert f \rVert_{L^2(Q_j)}^2
		\]
		for all $m \in \NN$ and $\alpha \in \{1,2\}^m$, where $C_B'(m)$ is defined in~\eqref{eq:magnetic_Bernstein_1}, and \emph{bad} otherwise.
	\end{definition}

	\begin{lemma}
		\label{lem:good_bad}
		Let $E, B > 0$ and $f \in \ran \mathbf{1}_{(- \infty, E]}(H_B)$. 
		Then
		\[
		\sum_{j \colon Q_j \mathrm{good}}
		\lVert f \rVert_{L^2(Q_j)}^2
		\geq
		\frac{1}{2}
		\lVert f \rVert_{L^2(\RR^2)}^2.
		\]
	\end{lemma}
	
	\begin{proof}
	Using the definition of badness and Theorem~\ref{thm:magnetic_Bernstein_2}, we estimate
	\begin{align*}
		\sum_{j \colon Q_j \text{bad}}
		\lVert f \rVert_{L^2(Q_j)}^2
		&\leq
		\frac{1}{4}
		\sum_{j \colon Q_j \text{bad}}
		\sum_{m = 0}^\infty
		\sum_{\alpha \in \{1,2\}^m}
		\frac{1}{4^m C_B'(m)}
		\lVert
		\partial^\alpha \lvert f \rvert
		\rVert_{L^1(Q_j)}^2
		\\
		&\leq
		\frac{1}{4}
		\sum_{m = 0}^\infty
		\sum_{\alpha \in \{1,2\}^m}
		\frac{1}{4^m C_B'(m)}
		\lVert
		\partial^\alpha \lvert f \rvert^2 
		\rVert_{L^1(\RR^2)}
		\\
		&\leq
		\frac{1}{4}
		\sum_{m = 0}^\infty
		\sum_{\alpha \in \{1,2\}^m}
		\frac{1}{4^m}
		\lVert f \rVert_{L^2(\RR^2)}^2
		= 
		\frac{1}{2}
		\lVert f \rVert_{L^2(\RR^2)}^2.
		\qedhere
	\end{align*}

	\end{proof}

	\subsection{Proof of Theorem~\ref{thm:main}}

	By Lemma~\ref{lem:analytic}, $\lvert f \rvert^2$ is analytic in $\RR^2$ with an analytic extension $\Phi$ to $\CC^2$.
	We may assume without loss of generality that $f$ does not vanish on $\RR^2$ and thus also, by analyticity, on none of the $Q_j$.
	Let $Q = Q_j$ be a good rectangle. 
	Then, we claim that there exists a point $x_0 \in Q$ such that for all $m \in \NN$ and all $\alpha \in \{1,2\}^m$ one has
	\begin{align}
	\label{eq:good_point}
	\lvert \partial^\alpha \lvert f \rvert^2 (x_0) \rvert
	\leq
	\frac{8^{m+1}
	C_B'(m)
	\lVert f \rVert_{L^2(Q)}^2}{\vol (Q)}
	.
	\end{align}
	Indeed, if there was no such point then for all $x \in Q_j$
	\[
	\frac{\lVert f \rVert_{L^2(Q)}^2}{\vol (Q)}
	<
	\sum_{m = 0}^\infty
	\sum_{\alpha \in \{1,2\}^m}
	\frac{1}{8^{m+1} C_B'(m)}
	\lvert \partial^\alpha \lvert f \rvert^2 (x) \rvert.
	\]
	But then, integration over $x \in Q_j$ and the definition of good rectangles would imply
	\[
	\lVert f \rVert_{L^2(Q)}^2
	<
	\sum_{m = 0}^\infty
	\sum_{\alpha \in \{1,2\}^m}
	\frac{1}{8^{m+1} C_B'(m)} 
	\lVert \partial^\alpha \lvert f \rvert^2 \rVert_{L^1(Q_j)}
	\leq
	\sum_{m = 0}^\infty	
	\frac{1}{2^{m+1}}
	\lVert f \rVert_{L^2(Q_j)}^2
	=
	\lVert f \rVert_{L^2(Q_j)}^2,
	\]
	a contradiction.
	This shows the existence of $x_0$ as in~\eqref{eq:good_point}.
	In particular, for every $z \in D_{(5 \ell_1, 5 \ell_2)} + x_0$
	\begin{align*}
		\lvert
		\Phi(z)
		\rvert
		&\leq
		\sum_{m = 0}^\infty
		\sum_{\alpha \in \{1,2\}^m}
		\frac{\lvert \partial^\alpha \lvert f \rvert^2 (x_0) \rvert}{m!} \lvert z - x_0 \rvert^\alpha
		\\\
		&\leq
		\sum_{m = 0}^\infty
		\sum_{\alpha \in \{1,2\}^m}
		\frac{8^{m+1} C'_B(m)}{m!}
		(5 \ell)^\alpha
		\frac{\lVert f \rVert_{L^2(Q)}^2}{\vol (Q)}
		\leq
		8
		\frac{\lVert f \rVert_{L^2(Q)}^2}{\vol (Q)}
		\sum_{m \in \NN} \frac{(40\lvert \ell \rvert_1)^m
		C'_B(m)}{m!}.
	\end{align*}
	We can therefore apply Lemma~\ref{lem:local_estimate} with $g = \lvert f \rvert^2$, $G = \Phi$, and, recalling $C_B'(m) = 2^{3m/2} (E + B m)^{m/2} \leq 3^m (E + B m)^{m/2}$, 
	\begin{align*}
	M_\varphi 
	&=
	8
	\sum_{m = 0}^\infty \frac{(40 \lvert \ell \rvert_1 )^m
	C'_B(m)}{m!}
	\leq
	8
	\sum_{m = 0}^\infty \frac{(120 \lvert \ell \rvert_1)^m (\sqrt{E} + \sqrt{B m})^{m}}{m!}
	\\
	&\leq
	8
	\sum_{m = 0}^\infty
	\frac{ (240 \lvert \ell \rvert_1 \sqrt{E})^m}{m!}
	+
	8
	\sum_{m = 0}^\infty
	\frac{ (240 \lvert \ell \rvert_1 \sqrt{B m})^m}{m!}
	.
	\end{align*}
	where we used $(a + b)^m \leq 2^m (a^m + b^m)$.
    Now, note that for all $s \geq 0$
    \begin{equation}
     \label{eq:lemma_with_s}
    \begin{aligned}
	\sum_{m = 0}^\infty \frac{(s \sqrt{m})^{m}}{m!} 
	=& 
	\sum_{k=0}^\infty s^{2k} \left( \frac{ \sqrt{2k}^{2k} }{(2k)!} + s \frac{ \sqrt{2k+1}^{2k+1}}{(2k+1)!} \right) 
	\leq
	(1 + s)
	\sum_{k = 0}^\infty
	\frac{(s \sqrt{2k})^{2 k}}{(2 k)!}
	\\
	&=
	(1 + s)
	\sum_{k = 0}^\infty
	\frac{(2 s^2)^k k^k}{(2k)!}
	\leq
	\exp(2 s^2 + s)
	\end{aligned}
	\end{equation}
	where we used $1 + s \leq \exp(s)$, and $(2k)! \geq k^k k!$ in the last step.
	Using~\eqref{eq:lemma_with_s} with $s = 240 \lvert \ell \rvert_1 \sqrt B$, we further estimate
    \begin{align*}
     M_\varphi
     &\leq
     8
     \exp \left( 240 \lvert \ell \rvert_1 \sqrt{E} \right)
     +
     8
     \exp \left( 240 \lvert \ell \rvert_1 \sqrt{B} + 2 \cdot 240^2 \lvert \ell \rvert_1^2 B \right)
     \\
     &\leq
     16
     \exp \left( 2 \cdot 240^2 \left( \lvert \ell \rvert_1 \sqrt{E} + \lvert \ell \rvert_1 \sqrt{B} + \lvert \ell \rvert_1^2 B \right) \right), 
    \end{align*}
    whence
	\[
	\ln M_\varphi
	\leq
	\ln 16
	+
	240 \lvert \ell \rvert_1 \sqrt{E}
	+
	2 \cdot 240^2 \left(\lvert \ell \rvert_1 \sqrt{B} + \lvert \ell \rvert_1^2 B \right).
	\]
	Therefore, we obtain for every good rectangle $Q_j$
	\begin{align*}
	\lVert f \rVert_{L^2(Q_j)}^2
	&\leq
	2
	\left(
	\frac{48 \pi \diam (A(Q_j))^2}{\vol (A(Q_j \cap S))}
	\right)^{C_2 + C_3 \lvert \ell \rvert_1 \sqrt{E} + C_4 \left( \lvert \ell \rvert_1 \sqrt{B} + \lvert \ell \rvert_1^2 B \right)}
	\lVert f \rVert_{L^2(Q_j \cap S)}^2.
	\end{align*}
	Choose the linear bijection $A$ to map every $Q_j$ to a square of unit length such that
	\[
	\frac{48 \pi \diam (A(Q_j))^2}{\vol (A(Q_j \cap S))}
	\leq
	\frac{96 \pi}{\rho}.
	\]
	Finally, summing over all good rectangles and using Lemma~\ref{lem:good_bad} we have
	\begin{align*}
	\lVert f \rVert_{L^2(\RR^2)}^2
	&\leq
	2
	\sum_{j \colon Q_j \text{good}}
	\lVert f \rVert_{L^2(\RR^2)}^2
	\leq
	\sum_{j \colon Q_j \text{good}}
	4
	\left(
	\frac{96 \pi}{\rho}
	\right)^{C_2 + C_3 \lvert \ell \rvert_1 \sqrt{E} + C_4 \left( \lvert \ell \rvert_1 \sqrt{B} + \lvert \ell \rvert_1^2 B \right)}
	\lVert f \rVert_{L^2(Q_j \cap S)}^2
	\\
	&\leq
	4
	\left(
	\frac{C_1}{\rho}
	\right)^{C_2 + C_3 \lvert \ell \rvert_1 \sqrt{E} + C_4 \left( \lvert \ell \rvert_1 \sqrt{B} + \lvert \ell \rvert_1^2 B \right)}
	\lVert f \rVert_{L^2(S)}^2.
	\end{align*}
	Using $\rho \leq 1$ to absorb the prefactor $4$ into the constant $C_2$, and using $B \leq E$ ($H_B$ has no spectrum below $B$) to absorb the term $\lvert \ell \rvert_1 \sqrt{B}$ into $\lvert \ell \rvert_1 \sqrt{E}$, we obtain the statement. 
	\qed

\section{Bounded domains}	
	\label{sec:finite_volume}
	
	In this section, we define finite-volume restrictions of $H_B$ onto rectangles, and indicate necessary modifications to the proof of Theorem~\ref{thm:main} in order to prove Theorem~\ref{thm:main_bounded_domain}.

	Let the \emph{magnetic translations} be defined by
	\[
		\left(
			\Gamma_{y} 
		\right)_{y \in \RR^2}
		\colon
		L^2(\RR^2) \to L^2(\RR^2)
		,
		\quad
		\left(
			\Gamma_{y} f 
		\right) (x)
		=
		\euler^{i \frac{B}{2} ( y_2 - y_1)}
		f(x - y)
	\]
	where $y = (y_1, y_2)$.
	This is a  family of unitary operators, where in contrast to usual translations on $\RR^2$, they only form a commutative group if we restrict them to vectors $y$ satisfying the so-called \emph{integer flux condition}
	\begin{equation}
	\label{eq:integer_flux_condition}
	B (y_2 - y_1) \in 2 \pi \ZZ.
	\end{equation}
	Following~\cite{Sjoestrand-91}, we define 
	\begin{multline*}
	\mathcal{H}_{B,\mathrm{loc}}^m (\RR^2)
	:=\\
	=\left\{
		f \in L^2_{\mathrm{loc}}(\RR^2)
		\colon
		\tildel_{\alpha_1} \dots \tildel_{\alpha_p} f \in L^2_{\mathrm{loc}}(\RR^2)
		\
		\forall \alpha = (\alpha_1, \dots, \alpha_p) \in \{1, 2 \}^p, p \leq m
	\right\},
	\end{multline*}
	as well as restrictions of these spaces to boxes
	\[
	\mathcal{H}_B^m (\Lambda_L)
	:=
	\left\{
		f \mid_{\Lambda_L}
		\colon
		f \in \mathcal{H}_{B,\mathrm{loc}}^m (\RR^2)
	\right\},
	\]
	and their ``periodic'' versions
	\[
	\mathcal{H}_{B, \mathrm{per}}^m (\Lambda_L)
	:=
	\left\{
		f \mid_{\Lambda_L}
		\colon
		f \in \mathcal{H}_{B,\mathrm{loc}}^m (\RR^2)\
		\text{with $\Gamma_y f = f$ for all $y$ satisfying~\eqref{eq:integer_flux_condition}}
	\right\}.
	\]
	Functions in $\mathcal{H}_{B, \mathrm{per}}^m (\Lambda_L)$ satisfy periodic boundary conditions where the usual periodicity has been replaced by invariance under magnetic translations.
	Now, a local version of the Landau operator can be defined as the self adjoint operator $H_{B,L}$ in the Hilbert space $L^2(\Lambda_L)$ with domain
	\[
	\mathcal{D}(H_{B,L}) = \mathcal{H}_{B, \mathrm{per}}^2(\Lambda_L)
	\]
	and acting as 
	\[
	H_{B,L} = \left( i \nabla + \frac {B}{2} \begin{pmatrix} - x_2 \\ x_1 \end{pmatrix} \right)^2
	.
	\]
	In particular, if~\eqref{eq:integer_flux_condition} is satisfied, then $\sigma(H_L)$ coincides with $\sigma(H) = \{ B, 3B, \dots \}$.

Let us now indicate which modifications are necessary for Theorem~\ref{thm:main_bounded_domain}.
		Large parts of the proof of Theorem~\ref{thm:magnetic_Bernstein} (the magnetic Bernstein inequalities on $\RR^2$) are analogous but we need to justify that when performing integration by parts, boundary terms will disappear. 
		This was obvious for Schwarz functions on $\RR^2$, but one needs a reasoning in the finite volume case.
	Due to the definition of $\mathcal{D}(H_L)$, it follows that one can use integration by parts for the magnetic derivatives $\tildel_1, \tildel_2$ of functions $f \in \mathcal{D}(H_L)$. 
	However, we also need integration by parts for higher order magnetic derivatices.
	For this purpose, the key is to observe that
	\[
	\mathcal{D}(H_{B,L}^k)
	\subseteq 
	\mathcal{H}_{B, \mathrm{per}}^{2 k}(\Lambda_L)
	\quad
	\text{for all $k \in \NN$}.
	\]
	This can for instance be seen by unitarity of the Floquet transform which pointwise maps $\mathcal{H}^{2k}_B(\RR^2)$ to $\mathcal{D}(H_{B,L}^k) (\Lambda_L)$ and which maps the domain of $H_B^k$ onto the one of $H_{B,L}^k$, cf.~\cite{Sjoestrand-91}.
	This justifies to also use integration by parts for the magnetic derivatives on any function
	\[
	f 
	\in 
	\ran \mathbf{1}_{(- \infty, E]} (H_{B,L})
	\subset
	\bigcap_{k \geq 1}
	\mathcal{D}(H_{B,L}^k) 
	\]
	and to prove finite-volume counterparts of
	Theorems~\ref{thm:magnetic_Bernstein} and~\ref{thm:magnetic_Bernstein_2}.
	
	The remaining steps of the proof of Theorem~\ref{thm:main_bounded_domain} follow with obvious modifications:
	\begin{itemize}
	\item
	Functions in $\ran \mathbf{1}_{(- \infty, E]} (H_{B,L})$ are a priori defined on $\Lambda_L$, but by magnetic translation they extend to functions on arbitrarily large boxes.
	\item
	Instead of Lemma~\ref{lem:analytic}, which shows analyticity of $\lvert f \rvert^2$ on $\RR^2$, it suffices to prove analyticity of their extension of boxes of fixed size, depending on $\ell = (\ell_1, \ell_2)$ -- the parameters from the definition of thickness.
	Instead of the global Sobolev estimate, we will therefore use a local Sobolev estimate in Lemma~\ref{lem:analytic} in order to infer analyticity, but the constant will remain uniformly bounded.
	\item
	We can no longer cover $\Lambda_L$ by mutually disjoint rectangles of side lengths $\ell_1, \ell_2$, but we can bound the overlap, see Remark~\ref{rem:overlap}.
	\end{itemize}

	The rest of the proof of Theorem~\ref{thm:main_bounded_domain} works verbatim as the proof of Theorem~\ref{thm:main}. 
	Note that the corresponding modifications for dealing with domains which are not $\RR^2$ itself are treated for example in the general setting in~\cite{EgidiS-21}, and in a more particular setting for the pure Laplacian in~\cite{EgidiV-20, Egidi-21}.

\section{Applications}
	\label{sec:applications}

\subsection{Controllability of the magnetic heat equation}

\label{sec:heat}

Consider the \emph{controlled heat equation with magnetic generator}
\begin{equation}
	\label{eq:magnetic_heat_controlled}
	\begin{cases}
		\frac{\partial}{\partial t} u + H_B u = \mathbf{1}_S f
		&
		\quad
		\text{in $\RR^2 \times (0,T)$},
		\\
		u(0) = u_0
		&
		\quad
		\in L^2(\RR^2).
	\end{cases}
\end{equation}
System~\eqref{eq:magnetic_heat_controlled} describes the diffusion of a (non-interacting) gas of charged particles in a plane, subject to a constant magnetic field, and controlled through an electric potential in $S \subseteq \RR^2$.

\begin{definition}
	System~\eqref{eq:magnetic_heat_controlled} is called \emph{null-controllable in time $T > 0$} if for every $u_0 \in L^2(\RR^2)$, there exists $f \in L^2((0, T) \times S)$ such that the solution of~\eqref{eq:magnetic_heat_controlled} satisfies $u(T) = 0$.
\end{definition}

The reason for restricting to the target state $u(T) = 0$ is that by linearity, this is equivalent to every state $u(T)$ in the range of the semigroup $(\euler^{- H_B t})_{t > 0}$ being reachable, the best notion of controllability one can hope for in parabolic systems.
By the classic Hilbert Uniqueness Method (HUM) due to Lions~\cite{Lions-88}, null-controllability is equivalent to \emph{final-state observability}, that is the estimate
\begin{equation}
	\label{eq:observability_estimate}
	\lVert 
	\euler^{- H_B T} u_0 \rVert_{L^2(\RR^2)}^2
	\leq
	C_{\obs}^2
	\int_0^T
	\lVert \euler^{- H_B t} u_0 \lVert_{L^2(S)}^2
	\drm t
	\quad
	\text{for all $u_0 \in L^2(\RR^2)$},
\end{equation}
and the least constant $C_\obs > 0$ in estimate~\eqref{eq:observability_estimate} is called the~\emph{control cost in time $T > 0$}.

Indeed, \eqref{eq:magnetic_heat_controlled} is an example of a wider class of parabolic systems with lower semibounded generator.
There is a strategy, combining spectral inequalities with the decay of the semigroup to prove an observability estimate: The so-called Lebeau-Robbiano-Strategy~\cite{LebeauR-95, LebeauL-12, TenenbaumT-11}.
In recent years, substantial effort has been devoted to deducing sharp estimates on the control cost in this method~\cite{TenenbaumT-11, Miller-04, Miller-10, NakicTTV-18, NakicTTV-20}.

\begin{proposition}[{Theorem 2.12 in~\cite{NakicSTTV-20}}]
	\label{prop:control_cost_estimate}
	Let $A \geq 0$ and let $X$ be a bounded, self-adjoint operator in a Hilbert space $\mathcal{H}$.
	Assume that for some $d_0, d_1 > 0$, one has the \emph{spectral inequality}
	\begin{equation}
		\label{eq:spectral_inequality_abstract}
		\lVert u_0 \rVert_{\mathcal{H}}^2
		\leq
		d_0
		\euler^{d_1 \sqrt{E}}
		\lVert 
		X 
		u_0
		\rVert_{\mathcal{H}}^2
		\quad
		\text{for all $E \geq 0$ and all $u_0 \in \ran \mathbf{1}_{(- \infty, E]}(A) $}.
	\end{equation}
	Then, for all $T > 0$, the \emph{observability inequality}
	\[
	\lVert 
	\euler^{- A T} u_0
	\rVert_{\mathcal{H}}^2
	\leq
	C_\obs^2
	\int_0^T 
	\lVert X \euler^{- t A} u_0
	\rVert_{\mathcal{H}}^2
	\
	\drm t
	\quad
	\text{for all $u_0 \in \mathcal{H}$}
	\]
	holds, where
	\[
	C_\obs^2
	\leq
	\frac{C_5 d_0}{T}
	\left(
		2 d_0 \lVert X \rVert + 1
	\right)^{C_6}
	\exp
	\left(
		\frac{C_7 d_1^2}{T}
	\right)
	\]
	for universal constants $C_5, C_6, C_7> 0$.
\end{proposition}

	Combining this with Theorem~\ref{thm:main}, we obtain:

\begin{theorem}
	\label{thm:controlled_heat_1}
	Let $B \geq 0$ and let $S \subseteq \RR^d$ be $(\ell, \rho)$-thick.
	Then, system~\eqref{eq:magnetic_heat_controlled} is null-controllable in every time $T > 0$ with cost $C_\obs$ satisfying
	\begin{equation}
	\label{eq:estimate_control_cost}
	C_\obs^2
	\leq
	\frac{C}{T \rho^{C +C \lvert \ell \rvert_1^2 B}}
	\exp
	\left(
		\frac{\ln \left( \frac{C}{\rho} \right) C \lvert \ell \rvert_1^2}{T}
	- B T
	\right)
	\end{equation}
	where $C > 0$ is a universal constant.	
	\end{theorem}
	
The estimate~\eqref{eq:estimate_control_cost} on the control cost $C_\obs$ has an asymptotic behaviour which is known to be optimal for the negative Laplacian $- \Delta$, cf. the discussion in~\cite{NakicTTV-20} and references therein: 
\begin{itemize}
	\item
	As $T \to 0$, the expression $C_\obs$ behaves proportional to $T^{-1/2}$ if $S \subset \RR^2$ is dense, and proportionally to $\exp (C/T)$ otherwise.
	\item
	As $T \to \infty$, the cost decays proportionally to $\exp(- C T)$, as necessary when the generator has a positive of its spectrum.
	\item
	Finally, in the homogenization regime, where $\lvert \ell \rvert_1$ tends to zero at fixed $\rho$, and fluctuations within $S$ become small while there is a uniform lower bound on the relative density, the influence of $S$ and $B$ on $C_\obs$ vanishes.
\end{itemize}

	\begin{proof}[{Proof of Theorem~\ref{thm:controlled_heat_1}}]
	The constant in the spectral inequality of Theorem~\ref{thm:main} is of the form
\begin{align*}
	&\left(
	\frac{C_1}{\rho}
	\right)^{C_2 + C_3 \lvert \ell \rvert_1 \sqrt{E} + C_4 (\lvert \ell \rvert_1^2 B )}
	=
	\underbrace{
	\left(\frac{C_1}{\rho} \right)^{C_2 + C_4 \lvert \ell \rvert_1^2 B}}_{:= d_0}
	\cdot
	\exp
	\left(
	\underbrace{\ln \left(\frac{C_1}{\rho} \right)
	C_3 \lvert \ell \rvert_1}_{:= d_1} \sqrt{E}
	\right).
\end{align*}
	Applying Proposition~\ref{thm:main} with $A := H_B$ and $X := \mathbf{1}_S$ for an $(\ell, \rho)$-thick $S \subseteq \RR^d$, we find that~\eqref{eq:magnetic_heat_controlled} is null-controllable in every time $T > 0$ with control cost satisfying	
	\begin{align*}
	C_\obs^2
	&\leq
		\frac{C_4}{T}
		\left( \frac{2 C_1 + 1}{\rho} \right)^{C_2 (C_5 + 1) + C_4 (C_5 + 1) \lvert \ell \rvert_1^2 B}
	\cdot
	\exp
	\left(
		\frac{\ln \left(\frac{C_1}{\rho} \right)^2
	C_3^2 \lvert \ell \rvert_1^2}{T}
	\right).
	\\
	&=
	\frac{D_1}{T \rho^{D_2 + D_3 \lvert \ell \rvert_1^2 B}}
	\exp
	\left(
		\frac{\ln \left( \frac{D_4}{\rho} \right)^2 D_5^2 \lvert \ell \rvert_1^2}{T}
	\right)
	\end{align*}
	for universal constants $D_1$ to $D_5$.
	This yields the bound 
	\[
	C_\obs^2
	\leq
	\frac{C}{T \rho^{C +C \lvert \ell \rvert_1^2 B}}
	\
	\exp
	\left(
		\frac{\ln \left( \frac{C}{\rho} \right) C \lvert \ell \rvert_1^2}{T}
	\right).
	\]
	We improve the large time behaviour of $C_\obs$ by using $\inf \sigma(H_B) = \lvert B \rvert \geq 0$, see for instance~\cite{NakicTTV-20}.
	Indeed, instead of controlling in the interval $[0,T]$, one can apply no control in the interval $[0, T/2]$ and then work with the new initial state $\euler^{- \frac{T}{2} H_B} u_0$ satisfying $\lVert \euler^{- \frac{T}{2} H_B} u_0 \rVert_{L^2(\RR^d)}^2 \leq \euler^{- T B} \lVert u_0 \rVert_{L^2(\RR^d)}^2$ in the interval $[T/2, T]$.
	Replacing $C$ again by $2 C$ in order to absorb the factor $\frac{1}{2}$ in $T$, we obtain the statement. 
\end{proof}

We next prove that thickness is also a sufficient criterion for observability, and thus null-controllability of the magnetic heat equation.
\begin{theorem}
	\label{thm:controlled_heat_2}
	If, for any $B \geq 0$, the observability estimate~\eqref{eq:observability_estimate}	holds, then $S \subset \RR^2$ must be thick.
\end{theorem}

\begin{proof}
	In the case of the free Laplacian, that is $B = 0$, this was proved independently in~\cite{EgidiV-18} and~\cite{WangWZZ-19}.
	For $B \neq 0$, we argue as in the proof of Theorem~\ref{thm:thickness_necessary_for_UCP}.
	If $S \subset \RR^2$ was not thick, then for every $n \in \NN$, there would exist $\{ y^{(n)} \}_{n \in \NN} \subset \RR^2$ such that $\vol ( B_n(y^{(n)}) \cap S) \leq \frac{1}{n}$ for all $n \in \NN$.
	Take $f_{y^{(n)}}$ defined as in~\eqref{eq:definition_f_y}, that is
	\[
	f_{y^{(n)}}(x)
	=
	\exp
	\left(
		- 
		\frac{B}{4} 
		\lvert x - y^{(n)} \rvert^2
		-
		i
		\frac{B}{2}
		\left(
		x_1 y^{(n)}_2 - x_2 y^{(n)}_1
		\right)
	\right).
	\]
	This is an eigenfunction to the eigenvalue $B$ satisfying $\lVert f_{y^{(n)}} \rVert_{L^2(\RR^2)}^2 = \frac{2 \pi}{B}$.
	In particular $\euler^{- H_B t} f_{y^{(n)}} = \euler^{- B t} f_{y^{(n)}}$, and $\lVert \euler^{- H_B T} f_{y^{(n)}} \rVert_{L^2(\RR^2)}^2 = \frac{2 \pi}{B} \euler^{- 2 B T}$.
	Hence,
	\begin{multline*}
	\int_0^T
	\lVert
		\euler^{- H_B t} f_{y^{(n)}} 
	\rVert_{L^2(S)}^2
	\drm t
	\leq
	\int_0^T
	\euler^{- 2 B t}
	\left(
		\lVert 
			f_{y^{(n)}} 
		\rVert_{L^2(S \cap B_n(y^{(n)})}^2
	+
		\lVert 
			f_{y^{(n)}} 
		\rVert_{L^2(B_n(y^{(n)})^c)}^2
	\right)
	\drm t
	\\
	\leq
	\
	T
	\left(
	\vol (S \cap B_n(y^{(n)}))
	+
	\int_n^\infty 
		\exp 
		\left( 
		- \frac{B}{2} r^2
		\right)
		r
		\drm r
	\right)
	\leq
	\frac{T}{n} 
	+
	\frac{T \exp (- \frac{B n^2}{2} )}{2}.
	\end{multline*}
	Since this tends to zero as $n \to \infty$, inequality~\eqref{eq:observability_estimate} cannot hold for any $C_{\obs} > 0$.
\end{proof}

We conclude that thickness is the \emph{optimal}, that is necessary and sufficient, geometric criterion for null-controllability of the magnetic heat equation -- the same as for the classic heat equation.

\subsection{Random Schr\"odinger operators}
	\label{sec:RSO}

Random Schr\"odinger operators are families of operators of the form
\[
	H_\omega
	=
	H_0
	+
	V_\omega,
	\quad
	\omega \in \Omega
\]
where $H_0$ is a self-adjoint operator in $L^2(\RR^d)$ (common cases include the negative Laplacian $- \Delta$, the negative Laplacian with periodic potential~\cite{BarbarouxCH-97, KirschSS-98, Klopp-99, Veselic-02, SeelmannT-20} or the Landau operator $H_B$~\cite{CombesHK-03, CombesHKR-04, CombesHK-07}), and $(V_\omega)_{\omega \in \Omega}$ is a random potential drawn from a probability space $(\Omega, \Sigma, \PP)$ with a $\sigma$-Algebra $\Sigma$ and a probability measure $\PP$ on $\Omega$, modeling a disordered solid.
The most common model in this context is the \emph{Alloy-type} or \emph{continuum Anderson model} in which case $\Omega = \RR^{\ZZ^2}$, $\Sigma = \times_{\ZZ^2} \mathcal{B}(\RR)$ where $\mathcal{B}(\RR)$ denotes the Borel-$\sigma$-algebra, and $\PP = \bigotimes_{\ZZ^2} \mu$ where $\mu$ is a non-trivial probability measure on $(\RR, \mathcal{B}(\RR))$ with bounded support, and 
\[
	V_\omega(x)
	=
	\sum_{j \in \ZZ^2}
	\omega_j
	u(x - j)
\]
for a compactly supported potential $0 \leq u \in L^\infty(\RR^2)$, which models the effect of a single atom.
Here, $\omega_j$ denotes the projection onto the $j$-th entry of an element $\omega \in \Omega = \RR^{\ZZ^2}$ so that the $\{ \omega_j \}_{j \in \ZZ^2}$ form a family of bounded, independent, and identically distributed random variables, distributed according to $\mu$.

Physical phenomena of interest in this context are \emph{Anderson localization} and \emph{Anderson delocalization}.
There are several notions of Anderson localization, the weakest one being the almost sure emergence of pure point spectrum with exponentially decaying eigenfunctions at certain energies, and a hierarchy of stronger notions of \emph{dynamic localization}, describing decay of correlations of functions of the operator $H_\omega$ in space.
Correspondingly, there  is a hierarchy of notions of~\emph{delocalization}, the strongest one being purely absolutely continuous spectrum and weaker ones involving dynamical notions and lower bounds on the decay of correlators in space.
We refer to the monographs~\cite{Stollmann-01, Veselic-08, AizenmanW-16} for a more comprehensive overview.

Whereas Anderson localization at extremal energies (the bottom of the spectrum or near band gaps) has been observed in a variety of models, delocalization is still mostly open and the Landau operator takes a particular role as the only known ergodic model of random Schr\"odinger operators on $\RR^2$ where -- under certain assumptions -- a localization-delocalization transition has been rigorously proved~\cite{GerminetKS-07}.
The latter crucially relies on the identification of a strict dichotomy of spectral regions of localization and delocalization~\cite{GerminetK-04}.

A central ingredient in proofs of localization (and thus, indirectly, of delocalization) are lower bounds of the form
\begin{equation}
	\label{eq:UCP}
\left\lVert
\sum_{j \in \ZZ^2}
u(\cdot - j)
f
\right\rVert_{L^2(\Lambda_L)}
\geq
C
\lVert
f
\rVert_{L^2(\Lambda_L)}
\quad
\text{for all $f \in \ran \mathbf{1}_{(- \infty, E]} (H_{0,L})$,}
\end{equation}
and for a family of $L$ of length scales, tending to infinity, where $H_{0,L}$ denotes the restriction of $H_0$ onto $L^2(\Lambda_L)$ with self-adjoint boundary conditions.

Clearly, if $\sum_{j \in \ZZ^2} u (\cdot - j)$ is uniformly positive on a suitable set $S \subset \RR^2$, then~\eqref{eq:UCP} is a direct consequence of Theorem~\ref{thm:main_bounded_domain}.
Indeed, such estimates have a tradition in the community on random Schr\"odinger operators where they are also referred to as \emph{quantitative unique continuation principles}.

However, so far, for a unique continuation estimate as in~\eqref{eq:UCP} to hold, one has usually had to assume that the function $\sum_{j \in \ZZ^2} u(x - j)$ be uniformly positive on an open set which had to be either periodic~\cite{CombesHK-03, CombesHK-07} or had to have at least some equidistributedness in space~\cite{RojasMolinaV-13, Klein-13, TaeuferV-15, NakicTTV-18, Taeufer-PhD}.
We can now relax this to merely positivity on a periodic set of positive measure (i.e. a periodic, thick set), which in light of~\cite{TaeuferV-21} seems to be the minimal assumption possible.
Furthermore, in the recent years, there has been interest in non-ergodic random Schr\"odinger operators~\cite{RojasMolina-12,RojasMolinaV-13,Klein-13,GerminetMRM-15, TaeuferT-18, MuellerRM-22, SeelmannT-20, TaeuferV-21}, a generalization which now also becomes accessible since we no longer rely on periodicity of $\sum_{j \in \ZZ^2} u(x - j)$.

In order to to illustrate that Theorem~\ref{thm:main_bounded_domain} yields an improvement of existing results, let us formulate a set of assumptions, inspired by common assumptions in the alloy-type model, cf. for instance~\cite[Section 1]{CombesHK-07}.

\begin{enumerate}[(i)]
	\item
	Let $B > 0$ and let the background operator be $H_B$.
	For $L > 0$ satisfying the integer flux condition let $H_L$ be the restriction of $H_B$ onto $L^2(\Lambda_L)$ with magnetic boundary conditions as defined in Section~\ref{sec:finite_volume}.
	\item
	Let $(u_j)_{j \in \ZZ^2}$ be a family of measurable functions satisfying $0 \leq \sum_{j \in \ZZ^2} u_j \leq 1$, and $\sum_{j \in \ZZ^2} u_j \geq \delta > 0$ on a thick set.
	\item
	Let $(\omega_j)_{j \in \ZZ^2}$ be a family of random variables, taking values in some interval $[m_0, M_0]$, $- \infty < m_0 < M_0 < \infty$.
	Call $\mu_j$ the conditional probability measure of $\omega_j$, conditioned on all other random variables $(\omega_k)_{k \neq j}$
	\[
	\mu_j
	( [E, E + \epsilon])
	=
	\mathbb{P}
	\left[
		\omega_j \in [E, E + \epsilon] \mid (\omega_k)_{k \neq j}
	\right],
	\]
	and define the \emph{conditional modulus of continuity}
	\[
	s(\epsilon)
	:=
	\sup_{j \in \ZZ^2}
	\mathbb{E}
	\left[
		\sup_{E \in \RR}
		\mu_j([E, E + \epsilon])
	\right],
	\quad
	\epsilon > 0
	\]
	where $\EE$ refers to the expectation with respect to the probability measure $\PP$.
\end{enumerate}

The novelty is assumption (ii) which no longer requires that $\sum_{j \in \ZZ^2} u_j$ be positive on a periodic, \emph{open} set.
Define the random Landau Hamiltonian as 
\[
	H_{B,\omega} = H_B + V_\omega,
	\quad
	V_\omega(x) = \sum_{j \in \ZZ^2} \omega_j u_j(x),
\]
and its restriction to boxes $\Lambda_L = (- \frac{L}{2}, \frac{L}{2})^2$ as
\[
	H_{B, \omega, L}
	:=
	H_{B,L}
	+
	V \mid_{\Lambda_L}
	\quad
	\text{with boundary conditions as defined in Section~\ref{sec:finite_volume}}.
\]
We then obtain a generalization of~\cite[Theorem 1.3]{CombesHK-07}, namely a Wegner estimate, optimal in energy and volume:

\begin{theorem}	
	\label{thm:Wegner}
	Assume Hypotheses (i)-(iii) above.
	Then, there is $L_0 > 0$ such that for all $E_0 \in \RR$, there exists $C_W > 0$, such that for all $E \leq E_0$, all $\epsilon \in (0,1]$, and all $L \geq L_0$ satisfying the integer flux condition
	\[
	B L \in 2 \pi \NN
	\]	
	 we have the Wegner estimate
	\begin{align*}
		\PP
		\left[
			\dist (\sigma (H_{B,\omega,L}), E) < \epsilon
		\right]
		&\leq
		\EE
		\left[ 
		\operatorname{Tr} \mathbf{1}_{[E - \epsilon, E + \epsilon ]} (H_{\omega, L})
		\right]
		\\
		&
		\leq 
		C_W
		s(2 \epsilon)
		L^2
	\end{align*}
	where $\dist (\sigma (H_{B,\omega,L}), E)$ denotes the distance between $\sigma (H_{B,\omega,L})$ and $E$. 
\end{theorem}

\begin{proof}
	The proof is completely analogous to the one in~\cite{CombesHK-07}, the only difference being that in our case, the potential
	\[
	\tilde V(x) 
	:=
	\sum_{j \in \ZZ^2}
	u(x - j)
	\]
	is no longer periodic and not uniformly positive on an open set, but merely on a thick set.
	But periodicity and openness were exactly used in~[Theorem 4.1]\cite{CombesHK-07} to prove
	\begin{equation}
	\label{eq:UCP_quadratic form}
	\Pi_{n,L} 
	\tilde V \mid_{\Lambda_L}
	\Pi_{n,L} 
	\geq
	C
	\Pi_{n,L}
	\end{equation}
	in the sense of quadratic forms, where $\Pi_{n,L} = \mathbf{1}_{\{(2 n + 1) B\} } (H_{B,L})$ is the spectral projector onto the $n$-th Landau level, that is onto the infinitely degenerate eigenvalue at energy $(2 n - 1) B$.
	But in light of Assumption (ii) above,~\eqref{eq:UCP_quadratic form} in our situation is an immediate consequence of Theorem~\ref{thm:main_bounded_domain}.
	For more details, we also refer to~\cite{TaeuferV-21}, where the corresponding argument is outlined in the case where the background operator is the free Laplacian.
\end{proof}

If the random family of operator $(H_\omega)_{\omega \in \Omega}$ is ergodic, its integrated density of states (IDS) 
	\begin{equation}
	\label{eq:def_IDS}
	N(E)
	:=
	\lim_{L \to \infty}
	\frac{\operatorname{Tr} \mathbf{1}_{(- \infty, E]}(H_{\omega}\mid_{\Lambda_L})}{\vol (\Lambda_L)},
	\end{equation}
	where $H_{\omega}\mid_{\Lambda_L})$ denotes a restriction of $H_\omega$ onto $L^2(\Lambda_L)$ with self-adjoint boundary conditions,
	exists almost surely.
	As a corollary, we obtain in this case the analogon of~\cite[Theorem 1.2]{CombesHK-07}, namely regularity of the integrated density of states:

\begin{corollary}
	Assume Hypotheses (i)--(iii) above and assume that the IDS exists almost surely for the family $(H_\omega)_{\omega \in \Omega}$.
	Then, for all $E_0 \in \RR$, there is $C > 0$ such that for all $E \leq E_0$ and all $\epsilon  \in (0,1]$, the IDS, defined in~\eqref{eq:def_IDS}, satisfies
	\[
		0 \leq N(E + \epsilon) - N(E) \leq C s(\epsilon).
	\] 
	In particular, if all $\omega_j$ are independent and identically distributed with bounded density, then the IDS is locally Lipschitz continuous.
\end{corollary}

Finally, note that Wegner estimates as in Theorem~\ref{thm:Wegner} are one important ingredient in so-called \emph{multiscale analysis} proofs of localization, the other central ingredient being \emph{initial length scale estimates}, see~\cite{Stollmann-01, GerminetK-01, GerminetK-03}.
Initial length scale estimates can for instance be inferred from Lipshitz tails, that is the fact that the IDS exponentially decays near its minimum, as derived in~\cite{KloppR-06}.
Theorem~4.1 (iii) in~\cite{KloppR-06} states such a lower bound under the hypothesis
\[
	u (x) \geq C\ \mathbf{1}_{\lvert x - x_0 \rvert < \epsilon} (x)
	\quad	
	\text{for some $x_0 \in \RR^2$, $C,\epsilon > 0$}.
\]
A closer inspection of the proof of said theorem yields that it essentially relies on lower bounds of the form
\[
	\lVert V_\omega \mid_{\Lambda_L} \psi \rVert_{L^2(\Lambda_L)}^2
	\geq
	C
	\lVert \psi \rVert_{L^2(\Lambda_L)}^2
	\quad
	\text{for all $\psi \in \mathbf{1}_{ \{B\} }(H_{B,L})$}
\]
for configurations $\omega$ with sufficiently high probability, cf.~\cite[Estimate (4.29)]{KloppR-06}.
This can be readily replaced by Theorem~\ref{thm:main_bounded_domain}.
In conclusions, by combining Theorem~\ref{thm:Wegner} with an initial scale estimate, derived from Theorem~\ref{thm:main_bounded_domain} and the method of proof of~\cite{KloppR-06} in the bootstrap multiscale analysis, one infers:

\begin{corollary}
	Let $0 \leq u \leq 1$ be measurable with non-empty, compact support.
	Let $(\omega_j)_{j \in \ZZ^2}$ be a family of independent and identically distributed random variables with bounded support, a bounded density $\rho$, and $\inf \operatorname{supp} \rho = 0$.
	Then, there is $\epsilon > 0$, such that the family of operators
	\[
	H_{B,\omega} 
	:=
	H_B
	+
	\sum_{j \in \ZZ^2}
	\omega_j 
	u(\cdot - j)
	\]
	exhibits strong dynamical localization in Hilbert Schmidt norm (and thus all other, weaker forms of Anderson localization) in the interval $[B, B + \epsilon]$.
\end{corollary}

The novelty is that the support of $u$ now no longer needs to be open, which seems to be the minimal assumption necessary.

\appendix

\section{Proof of Lemma~\ref{lem:one-dimensional} via Remez inequality}
\label{app:Remez}

For convenience and the sake of self-containedness, we provide  a proof of Lemma~\ref{lem:one-dimensional}. 
The version given here is essentially Lemma 1 in~\cite{Kovrijkine-00}.
It relies on the following variant of the Remez inequality for polynomials, which can be inferred from ~\cite[Theorem 5.1.1]{BorweinE-95}.
\begin{lemma}[Remez inequality]
	\label{lem:Remez}
	Let $P \colon \CC \to \CC$ be a polynomial of degree $n \in \NN$. 
	Then, for any measurable $E \subset [0,1]$ with positive measure
	\begin{equation}
		\label{eq:Remez}
		\sup_{t \in [0,1]} \lvert P(t) \rvert
		\leq
		\left(
			\frac{4}{\vol (E)}
		\right)^n
		\sup_{t \in E} \lvert P(x) \rvert.
	\end{equation}
\end{lemma}

Recall that $D_r \subset \CC$ denotes the complex polydisc with radius $r > 0$, centered at $0$.

\begin{proof}[Proof of Lemma~\ref{lem:one-dimensional}]
	The function $\varphi$ is not the zero function, so it has a finite number of zeroes in $D_2$, which we denote by $w_1, \dots, w_n$ (counting multiplicities).
	Define
	\[
	g(z)
	:=
	\varphi (z)
	\cdot
	\prod_{k = 1}^n
	\frac{4 - \overline w_k z}{2 (w_k - z)}
	=
	\varphi(z)
	\cdot
	\frac{Q(z)}{P(z)}
	\]
	where $Q$ and $P$ are polynomials.
	We have $\lvert g(0) \rvert \geq 1$ and $\max_{z \in D_2} \lvert g(z) \rvert \leq \max_{z \in D_2} \lvert \varphi(z) \rvert \leq M_\varphi$ by the maximum principle since the Blaschke product
	\[
	\prod_{k = 1}^n \frac{2 ( w_k - z)}{4 - \overline w_k z}
	=
	\frac{P(z)}{Q(z)}
	\]
	has modulus one on the boundary of $D_2$.
	Thus, $g$ is an analytic function without zeroes in $D_2$, and the function $\ln M_\varphi - \ln \lvert g(z) \rvert$ is positive and harmonic in $D_2$.
	By Harnack's inequality
	\[
	\max_{z \in D_1}
	\left(	
	\ln M_\varphi - \ln \lvert g(z) \rvert
	\right)	
	\leq
	\frac{1 + \frac{1}{2}}{1 - \frac{1}{2}}
	\left(
	\ln M_\varphi
	- 
	\ln \lvert g(0) \rvert
	\right)
	\leq
	3 \ln M_\varphi,
	\]
	whence in particular
	\[
	\min_{z \in D_1} \lvert g(z) \rvert \geq M_\varphi^{-2},
	\quad
	\text{and}
	\quad
	\frac{\max_{t \in [0,1]} \lvert g(t) \rvert}{\min_{t \in [0,1]} \lvert g(t) \rvert} 
	\leq 
	M_\varphi^3.
	\]
	Likewise, for every $k \in \{1, \dots, n\}$, the function $z \mapsto (4 - \overline{w_k} z)$ is analytic in $D_1$ without zeroes.
	By the maximum principle $z \mapsto \lvert 4 - \overline{w_k} z \rvert$ takes its maximum and minimum in $D_1$ on the boundary where
	\[
	2 \leq \lvert 4 - \overline{w_k} z \rvert \leq 6.
	\]
	This implies
	\[
	\frac{\max_{t \in [0,1]} \lvert Q(t) \rvert}
	{\min_{t \in [0,1]} \lvert Q(t) \rvert}
	\leq
	\prod_{k = 1}^n
	\frac{\max_{z \in D_1} \lvert 4 - \overline{w_k} z \rvert}
	{\min_{z \in D_1} \lvert 4 - \overline{w_k} z \rvert}
	\leq 
	3^n.
	\]
	Combining this with Lemma~\ref{lem:Remez}, we find
	\begin{align*}
		\sup_{t \in [0,1]}
		\lvert \varphi (x) \rvert
		&\leq
		\max_{t \in [0,1]} \lvert g(x) \rvert
		\frac{\max_{t \in [0,1]} \lvert P(x) \rvert}
		{\min_{t \in [0,1]} \lvert Q(x) \rvert}
		\\
		&\leq
		M_\varphi^3
		\cdot
		\left( \frac{12}{\vol (E)} \right)^n
		\min_{t \in [0,1]} \lvert g(x) \rvert
		\frac{\sup_{t \in E} \lvert P(x) \rvert}
		{\max_{t \in [0,1]} \lvert Q(x) \rvert}
		\\
		&\leq
		M_\varphi^3
		\cdot
		\left( \frac{12}{\vol (E)} \right)^n
		\sup_{t \in E}
		\lvert \varphi (x) \rvert.
	\end{align*}
	Finally, by Jensen's formula, the number $n$ of zeroes of $\varphi$ in $D_2$ is bounded by $\frac{\ln M_\varphi}{\ln 2}$.
	Thus
	\[
	\sup_{t \in [0,1]}
	\lvert \varphi(x) \rvert
	\leq
	M_\varphi^3 \left( \frac{12}{\vol (E)} \right)^{\frac{\ln M_\varphi}{\ln 2}}
	\sup_{t \in E} \lvert \varphi(x) \rvert
	\leq
	\left( \frac{12}{\vol (E)} \right)^{2 \frac{\ln M_\varphi}{\ln 2}}
	\sup_{t \in E} \lvert \varphi(x) \rvert
	.
	\qedhere
	\]
\end{proof}



\end{document}